\documentclass[12pt,a4paper,reqno]{amsart}
\usepackage[utf8]{inputenc}
\usepackage[T1]{fontenc}
\usepackage{amsmath}
\usepackage{amsthm}
\usepackage{amssymb}
\usepackage[abbrev]{amsrefs}
\usepackage{mathrsfs}
\usepackage[dvipsnames]{xcolor}
\usepackage{bm}
\usepackage[shortlabels]{enumitem}
\usepackage{bbm}
\usepackage{hyperref}

\makeatletter
\newcommand{\lenos}{\tagsleft@true\let\veqno\@@leqno}
\newcommand{\reqnos}{\tagsleft@false\let\veqno\@@eqno}
\reqnos
\makeatother

\AtBeginDocument{\def\MR#1{}}

\newtheorem{thm}{Theorem}[section]
\newtheorem{theorem}[thm]{Theorem}
\newtheorem{corollary}[thm]{Corollary}
\newtheorem{lemma}[thm]{Lemma}
\newtheorem{proposition}[thm]{Proposition}

\newtheorem{definition}[thm]{Definition}
\newtheorem{remark}[thm]{Remark}
\newtheorem{question}[thm]{Question}
\newtheorem{example}[thm]{Example}

\numberwithin{equation}{section}
\allowdisplaybreaks

\newcommand{\FF}{\ensuremath{\mathbb{F}}}
\newcommand{\RR}{\ensuremath{\mathbb{R}}}
\newcommand{\NN}{\ensuremath{\mathbb{N}}}
\newcommand{\xx}{\ensuremath{\bm{x}}}

\newcommand{\ee}{\ensuremath{\bm{e}}}

\newcommand{\XX}{\ensuremath{\mathbb{X}}}
\newcommand{\XB}{\ensuremath{\mathcal{X}}}

\newcommand{\EE}{\ensuremath{\mathcal{E}}}
\newcommand{\Ind}{\ensuremath{\mathbbm{1}}}

\newcommand{\g}{\ensuremath{\bm{g}}}

\DeclareMathOperator{\sgn}{sign}
\DeclareMathOperator{\spn}{span}
\DeclareMathOperator{\supp}{supp}

\newcommand{\la}{\boldsymbol\lambda}


\usepackage[left=2.5cm,right=2.5cm,top=2cm,bottom=2cm]{geometry}
\begin{document}
	\title[Lebesgue-type estimates for greedy algorithms]{Lebesgue-type estimates for greedy algorithms in quasi-Banach spaces}
	\author[M. Berasategui]{Miguel Berasategui}

\address{Miguel Berasategui
	\\
	UBA - Pab I, Facultad de Ciencias Exactas y Naturales \\ Universidad de Buenos Aires \\ Buenos Aires 1428, Argentina}
\email{mberasategui@dm.uba.ar}

	\author[P. M. Bern\'a]{Pablo M. Bern\'a}
	\address{Pablo M. Bern\'a\\
		Departamento de Matemáticas, CUNEF Universidad\\
		 Madrid 28040, Spain.}
	\email{pablo.berna@cunef.edu}
	
		\author[H. V. Chu]{H\`ung Vi\d{\^e}t Chu}
	\address{H\`ung Vi\d{\^e}t Chu\\
		Department of Mathematics, Washington and Lee University, Lexington, VA 24450, USA.}
	\email{hchu@wlu.edu}
	
	\author[A. García]{Andrea García}
	\address{Andrea Garc\'{i}a\\
	Universidad San Pablo CEU, CEU Universities, Madrid 28003, Spain, and 		Departamento de Matemáticas, CUNEF Universidad, Madrid 28040, Spain.}
	\email{andrea.garciapons@usp.ceu.es}

	\begin{abstract}
		We continue the study of Lebesgue-type parameters for various greedy algorithms in quasi-Banach spaces. First, we introduce a parameter that can be used with the quasi-greedy parameter to obtain the exact growth of the Lebesgue parameter for strong partially greedy bases. Second, we establish a new upper bound for the Lebesgue parameter for semi-greedy bases using the quasi-greedy and the squeeze symmetry parameters. Finally, we answer several open questions regarding the optimal power in various bounds proved in [F. Albiac, J. L. Ansorena, and P. M. Bern\'{a}, New parameters and Lebesgue-type estimates in greedy approximation, \textit{Forum Math. Sigma} \textbf{10} (2022), 1--39].
	\end{abstract}
	
	
	
	\thanks{The first author was supported by the Grants ANPCyT PICT 2018-04104 and CONICET PIP 11220200101609CO. The second and fourth author were supported by the Grant PID2022-142202NB-I00 (Agencia Estatal de Investigación, Spain).}
	
	\maketitle

    \tableofcontents
	\section{Introduction and main results}\noindent
A \emph{quasi-Banach space} is a complete vector space $\XX$ over the  field $\FF=\mathbb R$ or $\mathbb C$ equipped with a \emph{quasi-norm}, i.e., a map $\|\cdot\|\colon \XX\to [0,\infty)$ that has all properties of a norm except that the triangle inequality is replaced by 
	\begin{equation}\label{defquasinorm}
	\|f+g\|\ \le\ K( \| f\| + \|g\|),\quad f,g\in \XX,
	\end{equation}
	for some $K\ge 1$ independent of $f$ and $g$. The smallest $K$ is called the \textit{modulus of concavity}. Given $0<p\le 1$, a \emph{$p$-Banach space} is a quasi-Banach space whose quasi-norm is $p$-subadditive:
	\[
	\| f+g\|^p \ \le\ \| f\|^p +\| g \|^p, \quad f,g\in\XX.
	\]
	It is trival that every $p$-Banach space is a quasi-Banach space with $K \le 2^{1/p-1}$, and conversely, by the Aoki-Rolewicz Theorem \cite{Ao,R}, every quasi-Banach space is a $p$-Banach space under an equivalent quasi-norm. 
	
	We let $\mathbb{X}$ denote separable $p$-Banach spaces $(0<p\le 1)$ (and thus, all quasi-norms are continuous). Let $\XB=(\xx_n)_{n=1}^\infty\subset\XX$ be a \emph{fundamental minimal system} for $\XX$, which is a sequence that satisfies the following conditions:
	\begin{itemize}
		\item[i)] $\overline{\spn}(\xx_n \colon n\in\NN)=\XX$, and
		\item[ii)] there is a unique sequence $\XB^{\ast}=(\xx_{n}^{\ast})_{n=1}^\infty$ in the dual space $\XX^{\ast}$ such that $(\xx_{n}, \xx_{n}^{\ast})_{n=1}^{\infty}$ is a biorthogonal system.
		\end{itemize}
	We refer to $\XB^{\ast}$ as the \emph{dual system} of $\XB$. 

	   When $\overline{\spn(\xx_n^*: n\in \NN)}^{w^*}=\XX^*$, $\XB$ is called a \emph{Markushevich basis for $\XX$}.	In this case, each element $f\in\XX$ can be represented by the unique formal series expansion $\sum_{n=1}^\infty \xx_n^*(f)\xx_n$. If additionally,  	
	\begin{itemize}
		\item[iii)]  the \textit{partial sum projections} $P_m : \XX \mapsto\XX$ with respect to $\XB$, given by
		$$f\ \mapsto\ P_{m}[\XB,\mathbb X](f)=P_m(f)\ :=\ \sum_{n=1}^{m} \xx_n^*(f)\, \xx_{n}, \quad f\in\XX,\, m\in\NN,$$
		are uniformly bounded, we say that $\XB$ is a \textit{Schauder} basis.
	\end{itemize}
 For the greedy algorithms to work, we need to assume that both $\XB$ and $\XB^*$ are bounded. For each $A\subseteq \NN$, we let $$\EE_A\ :=\ \{(\varepsilon_n)_{n\in A}\subset \mathbb{F}\,:\, |\varepsilon_n| = 1\mbox{ for all }n\in A\}.$$
 Given a fundamental minimal system $\XB=(\xx_n)_{n=1}^\infty$, a finite set $A\subseteq\NN$, and $\varepsilon=(\varepsilon_n)_{n\in A}\in\EE_A$, define
		\[
		\textstyle
		\Ind_{\varepsilon,A}[\XB,\XX]\ =\ \Ind_{\varepsilon,A}\ :=\ \sum_{n\in A} \varepsilon_n \xx_n \mbox{ and } \Ind_{A}[\XB,\XX]\ =\ \Ind_{A}\ :=\ \sum_{n\in A} \xx_n
		\]
and the projection operator
$$P_A[\XB, \XX](f)\ =\ P_A(f)\ :=\ \sum_{n\in A}\xx_n^*(f)\xx_n.$$
	
Since 1999, one of the most studied algorithms in the theory of nonlinear approximation is the Thresholding Greedy Algorithm introduced by Konyagin and Temlyakov \cite{KoTe1999}. For $f\in\XX$ and $m\in\mathbb N_0$, a \textit{greedy set} $A$ of $f$ of order $m$ is a collection of $m$ indices with $| \xx_n^*(f)|\ge | \xx_k^*(f)|$ for all $n\in A$ and for all $k\not\in A$. Then $P_A(f)$ is called a \textit{greedy sum} of $f$ of order $m$. Greedy sets are not necessarily unique when there are different coefficients having the same modulus. To resolve this issue, the Thresholding Greedy Agorithm (TGA) $(\mathcal G_m)_{m=1}^\infty$ uses the natural ordering of $\mathbb{N}$ to construct a unique greedy set of every order. 
In particular, for $f\in\XX$ and $m\in\mathbb N$, a greedy set $A_m(f)$ is defined recursively as follows: let
$$A_1(f)\ :=\ \min\{n\in \mathbb N\,:\, |\xx_n^*(f)| = \max_{j\in \mathbb N} |\xx_j^*(f)|\};$$
supposing that $A_{m-1}$ has been defined for some $m\ge 2$, we set
	$$k(m) \ :=\ \min\left\{ n\in\mathbb N \backslash A_{m-1}(f)\, :\, | \xx_n^*(f)|=\max_{j\not\in A_{m-1}(f)}| \xx_{j}^*(f)\right\}$$
	and $$A_m(f)\ =\ A_{m-1}(f)\cup \{ k(m)\}.$$ Then 
	$$\mathcal G_m[\XB,\XX](f)\ =\ \mathcal G_m(f)\ :=\ \sum_{n\in A_m(f)}\xx_n^*(f)\xx_n,\quad f\in\XX, m\in\mathbb N.$$

	As for every algorithm, a question raised in \cite{KoTe1999} was when $\mathcal{G}_m(f)\rightarrow f$ in norm for every $f\in \XX$. The convergence requirement turns out to be equivalent to the notion of quasi-greediness which is captured by the \textit{quasi-greedy parameter} $\mathbf g_m^c$:
	$$\mathbf g_m^c \ :=\ \sup_{k\le m}\sup_{f\neq 0}\frac{\| f-\mathcal G_k(f)\|}{\| f\|}.$$
	
	\begin{definition}\normalfont\cite{KoTe1999}
		A fundamental minimal system $\XB$ for a quasi-Banach space $\XX$ is \textit{quasi-greedy} if $C_q:=\sup_m \mathbf g_m^c<\infty$.
	\end{definition}

Wojtaszczyk \cite{Woj2000} proved that $\lim_{m\rightarrow\infty}\|f-\mathcal G_m(f)\| = 0$ for every $f\in \XX$ if and only if $\XB$ is quasi-greedy.

\begin{remark}\normalfont
A standard perturbation argument gives that $\mathbf g_m^c$ is the smallest constant $C$ such that
$$\| f-P_A(f)\| \ \le\ C\| f\|,\quad f\in\XX,\, A\, \text{greedy set of}\, f, | A|\le m.$$
Moreover, we define the parameter $\mathbf g_m$  as the smallest constant $C$ such that
$$\| P_A(f)\| \ \le\ C\| f\|,\quad f\in\XX,\, A\, \text{greedy set of}\, f, | A|\le m.$$
We have
\begin{equation}\label{quasi}
	\mathbf g_m \ \le\ (1+(\mathbf g_m^c)^p)^{1/p}\mbox{ and } \mathbf g_m^c \ \le\ (1+(\mathbf g_m)^p)^{1/p}.
\end{equation}
\end{remark}
	
Since 2011, many authors \cite{AAB2, BBG2017,BBGHO2018,DKO,GHO2013} have studied the performance of the TGA through the following parameters: for every $m\in \mathbb{N}$,  $\mathbf L_m[\XB,\XX]=\mathbf L_m$ and $\mathbf L_m^a[\XB,\XX]=\mathbf L_m^a$ are the smallest numbers such that for each $f\in\mathbb X$,
\begin{align*}\| f-P_A(f)\| &\ \le\ \mathbf L_m\inf\{\| f-y\|\, :\, |\supp(y)|\le| A|\}, \quad A \mbox{ greedy set of } f, | A|\le m, \mbox{ and }\\
\|f-P_A(f)\| &\ \le\ \mathbf L_m^a\inf\{ \|f-P_B(f)\|\, :\, |B|\le |A|\},\quad A \mbox{ greedy set of } f, |A|\le m.\end{align*}
When $\mathbf L_m =O(1)$, we say that $\XB$ is a \textit{greedy} basis \cite{AABW2021,KoTe1999}, which is characterized by unconditionality and democracy. 

\begin{definition}\normalfont
A fundamental minimal system $\XB$ for a quasi-Banach space $\mathbb X$ is said to be \textit{unconditional} if $\sup_m \mathbf k_m <\infty$, where $\mathbf k_m[\XB,\mathbb X]=\mathbf k_m$ is the smallest constant with
$$\| P_A(f)\|\ \le\ \mathbf k_m\| f\|,\quad f\in\mathbb X, |A|\le m.$$
Meanwhile, $\XB$ is \textit{democratic} if $\sup_m \boldsymbol{\mu}_m<\infty$, where 
$$\boldsymbol{\mu}_m[\XB,\mathbb X]\ =\ \boldsymbol{\mu}_m\ :=\ \sup_{| A|= | B|\le m}\frac{\|\Ind_{A}\|}{\| \Ind_B\|}.$$
\end{definition}

When $\mathbf L_m^a=O(1)$, $\XB$ is called \textit{almost-greedy} and is characterized by quasi-greediness and democracy \cite{AABW2021,DKKT2003}.

\begin{remark}\normalfont
Related to $\boldsymbol{\mu}_m$, we define
$$\boldsymbol{\mu}_m^d[\XB,\mathbb X]\ =\ \boldsymbol{\mu}_m^d\ :=\ \sup_{| A|= | B|\le m, A\cap B=\emptyset}\frac{\|\Ind_{A}\|}{\| \Ind_B\|}.$$
It is easy to see that $\boldsymbol{\mu}_m^d \le \boldsymbol{\mu}_m\le (\boldsymbol{\mu}_m^d)^2$.
\end{remark}

To study whether the greedy approximation always performs better than the convenient partial summations $(P_m)_{m=1}^\infty$, Dilworth et al.\ \cite{DKKT2003} introduced the notion of \textit{partially greedy Schauder bases} as Schauder bases with an absolute positive constant $C$ such that 
$$\| f-\mathcal G_m(f)\|\ \le\ C\| f-P_m(f)\|,\quad \forall f\in\XX, \forall m\in\mathbb N.$$
They established that a Schauder basis $\XB$ of a Banach space is partially-greedy in a Banach space if and only if $\XB$ is quasi-greedy and conservative, where $\XB$ is $D_c$-\textit{conservative} if
\begin{equation}\label{con}
\| \Ind_A\|\ \le\ D_c\| \Ind_B\|,\quad \forall | A|\le| B|, A<B.
\end{equation}
Here $A < B$ means that $a < b$ for every $a\in A$ and $b\in B$.  The same characterization holds for quasi-Banach spaces \cite{Berna2}. 

Berasategui et al.\ \cite{BBL} extended the partially greedy property to Markushevich bases by introducing the \textit{strong partially greedy} property. For $m\in \mathbb{N}$, the $m$\textsuperscript{th} \textit{strong residual Lebesgue parameter} is the smallest number $\mathbf L_m^s = \mathbf L_m^s[\XB, \mathbb{X}]$ such that   
\begin{equation}\label{partially}\|f - P_A(f)\|\ \le\ \mathbf L_m^s\inf_{k\le |A|}\|f-P_k(f)\|, \quad f\in \mathbb{X}, A\mbox{ greedy set of }f, |A|\le m.\end{equation}

\begin{definition}\normalfont
A fundamental minimal system $\XB$ for a quasi-Banach space $\mathbb X$ is \textit{strong partially greedy} if $\sup_m \mathbf L_m^s<\infty$.
\end{definition}

\begin{remark}\normalfont Every quasi-greedy fundamental minimal system is a Markushevich basis. 
(see \cite[Theorem 1]{Woj2000} and \cite[Theorem 4.1]{AABW2021}). Furthermore, every strong partially greedy fundamental minimal system is quasi-greedy, so every strong partially greedy fundamental minimal system must be Markushevich.
\end{remark}

To estimate $\mathbf L_m^s$, the authors of \cite{BBL} used the following parameter: for $m\in \mathbb{N}$, 
$$\omega_m\ :=\ \sup\left\{ \frac{\| f+t\Ind_{\varepsilon,A}\|}{\| f+t\Ind_{\eta,B}\|} \,:\, (A, B, f, \varepsilon, \eta)\in\mathcal F_m\right\},$$
where $(A, B, f, \varepsilon, \eta)\in\mathcal F_m$ if $| A|\le | B|\le m$, $A< m$, $A<\supp(f)\cup B$, $\supp(f)\cap B=\emptyset$, $\varepsilon\in\mathcal E_A$, $\eta\in\mathcal E_B$ and $| t|\ge \max_n|\xx_n^*(f)|$.

\begin{theorem}\cite[Proposition 1.13 and Theorem 1.14]{BBL}\label{f1}
Let $\XB$ be a fundamental minimal system for a Banach space $\XX$. Then for each $m\in \mathbb{N}$,
$$
\boldsymbol{g}_m^c\ \le\ \mathbf L_m^s,\quad \omega_m \ \le\ \max_{1\le k\le m}\mathbf L_m^s,\quad\mbox{ and }\quad \mathbf L_{m}^s\ \le\ \mathbf g_m \omega_m.
$$
\end{theorem}

Thanks to \cite[Example 3.4]{BBL}, we know that the bound $\mathbf{L}_m^s \approx \omega_m \mathbf{g}_m$ is not optimal in general. The first goal of this paper is to introduce a new parameter that, together with $\mathbf g_m$, gives the correct growth of $\mathbf{L}_m^s$.  For $m\in \mathbb{N}$, we define the \textit{conservative squeeze symmetry parameter} $\la_m^c  =\la_m^c[\XB,\XX]$ as the smallest constant $C$ such that
\begin{equation}\label{spg}
    \min_{n\in B}|\xx_n^*(f)|\|\Ind_{\varepsilon, A}\|\ \le\  C\| f\|,
\end{equation}
whenever $A\subseteq\mathbb N$, $\varepsilon\in\mathcal E_{A}$, and $B$ is a greedy set of $f\in\XX$ satisfying $A\le m$, $A< \supp(f)$, and $|A|\le |B|\le m$. 

The parameter $\la_m^c$ is inspired by the \textit{$m$\textsuperscript{th} squeeze symmetry parameter} $\la_m$, where $\la_m =\la_m[\XB,\XX]$ is the smallest constant $C$ such that
	$$\min_{n\in B}|\xx_n^*(f)|\|\Ind_{\varepsilon, A}\|\ \le\ C\| f\|,$$
	whenever $A\subseteq\mathbb N$, $\varepsilon\in\mathcal E_{A}$, and $B$ is a greedy set of $f\in\XX$ satisfying  $| B|=| A|\le m$. This parameter $\la_m$ appeared in \cite[Theorem 3.5]{AAB2}, which stated that 
    \begin{equation}\label{ch1}\mathbf L_m\ \approx\ \max\{\la_m, \mathbf k_m\}\end{equation}
    and thus, answered a question by Temlyakov \cite{Temlyakov2011} about the correct growth of $\mathbf L_m$.
        
    If we require further that $A\cap \supp(f)=\emptyset$ in the definition of $\la_m$, we obtain $\la_m^d$, called the \textit{$m$\textsuperscript{th} disjoint squeeze symmetry parameter}, and \cite[Theorem 4.2]{AAB2} gave
		\begin{equation}\label{ch2}
			\mathbf L_m^a\ \approx\ \max\{\la_m^d, \mathbf g_m\}.
		\end{equation}

Our first result establishes the analog of \eqref{ch1} and \eqref{ch2} for $\mathbf L_m^s$. 

\begin{theorem}\label{main1}
Let $\XB$ be a fundamental minimal system for a $p$-Banach space $\XX$. Then there are two positive constants $C_1$ and $C_2$ depending only on $p$ such that 
$$C_1 \mathbf L_m^s\ \le\ \max\{\la_m^c, \mathbf g_m\}\ \le\ C_2\mathbf L_m^s.$$
\end{theorem}
Different from \cite[Theorem 1.14]{BBL} and \cite[Theorem 4.2]{Berna2}, Theorem \ref{main1} provides the correct order of $\mathbf L_m^s$.

The second goal of the paper is to use $\mathbf g_m$ and $\la_m$ to bound the Lebesgue parameter for semi-greedy bases, whose definition is based on the Chebyshev Thresholding Greedy Algorithm (CTGA) \cite{DKK2003}. The CTGA 
produces approximations ($\mathcal{CG}_m(f))_{m=1}^\infty$ for each $f\in \XX$ such that
$$\| f-\mathcal{CG}_m(f)\|\ =\ \inf_{a_n\in\mathbb F}\left\|f-\sum_{n\in A_m(f)}a_n\xx_n\right\|, \quad m\in\mathbb{N}.$$ 
It follows from the definition that $\|f-\mathcal{CG}_m(f)\|\le \|f-\mathcal{G}_m(f)\|$, so the CTGA is an enhancement of the TGA. 

\begin{definition}\normalfont
A fundamental minimal system is \textit{semi-greedy} if 
$\sup_m \mathbf L_m^{ch}<\infty$, where $\mathbf L_m^{ch}[\XB,\XX]=\mathbf L_m^{ch}$ is the smallest constant such that for every $f\in\XX$,
$$\| f-\mathcal{CG}_m(f)\|\ \le\ \mathbf L_m^{ch}\inf\{ \| f-y\| : | \supp(y)|\le m\}.$$
\end{definition}

Dilworth et al. \cite{DKK2003} proved that in a Banach space with finite cotype, a Schauder basis is semi-greedy if and only if it is almost greedy. Later, Bern\'{a} \cite{Berna2019, Berna2020} removed the finite cotype requirement and extended the equivalence to weights. More recently, Berasategui and Lassalle proved that the equivalence still holds for the more general class of Markushevich bases in Banach spaces  \cite{BL2023a} and studied the corresponding weighted version and parameters \cite{BL2023b}. It is worth noting that without the Markushevich hypothesis, the equivalence may fail \cite{BL2023a}. The following gives an upper bound for $\mathbf L_m^{ch}$.

\begin{theorem}\cite[cf.\ Theorem 1.2]{BBGHO20}\label{oldbound}
Let $\XB$ be a Markushevich basis for a Banach space $\XX$. For all $m\in\mathbb{N}$,
$$\mathbf L_m^{ch}\ \le\ \mathbf g_{2m}^c + 4\mathbf g_m\tilde{\boldsymbol{\mu}}_m,$$
where 
$$\tilde{\boldsymbol{\mu}}_m \ =\ \sup_{\substack{| A|=| B|\le m\\ \varepsilon\in \mathcal E_A, \eta\in \mathcal E_B}}\frac{\| \Ind_{\varepsilon, A}\|}{\| \Ind_{\eta, B}\|}$$ is the \textit{super-democracy} parameter.
\end{theorem}

\begin{remark}\normalfont
We have the disjoint version of $\tilde{\boldsymbol{\mu}}_m$:
$$\tilde{\boldsymbol{\mu}}_m^d[\XB,\mathbb X]\ =\ \tilde{\boldsymbol{\mu}}_m^d\ :=\ \sup_{\substack{| A|=| B|\le m, A\cap B = \emptyset\\ \varepsilon\in \mathcal E_A, \eta\in \mathcal E_B}}\frac{\| \Ind_{\varepsilon, A}\|}{\| \Ind_{\eta, B}\|}.$$
Clearly, $\tilde{\boldsymbol{\mu}}_m^d \le \tilde{\boldsymbol{\mu}}_m\le (\tilde{\boldsymbol{\mu}}_m^d)^2$.
\end{remark}

Our second result bounds $\mathbf L_m^{ch}$ from above using $\mathbf g_m$ and $\la_m$. 

\begin{theorem}\label{cheby}
Let $\XB$ be a fundamental minimal system for a $p$-Banach space $\XX$. For $m\in \mathbb{N}$, we have
$$\mathbf L_m^{ch}\ \lesssim\ \begin{cases}\max\{ \mathbf g_m, \la_m\},&\mbox{ if } p =1;\\ 
\max\{ \mathbf g_m^{1+1/p}, \la_m\}, &\mbox{ if } p < 1.\end{cases}$$
\end{theorem}

\begin{remark}\rm \label{remarkimprovement}
We give an example of when the bound in Theorem \ref{cheby} is better than the one in Theorem \ref{oldbound}. Let $\XB$ be the summing basis of $\mathtt{c}_0$, i.e., for each $n$, $\xx_n:=\sum_{1\le j\le n}\ee_j$, where $(\ee_j)_{j\in \NN}$ is the canonical unit vector basis of $\mathtt{c}_0$. For all $m\in \NN$, we have 
$$
\left\|\sum_{n=1}^{2m}(-1)^{n} \xx_n \right\|\ =\ 1,\quad \left\|\sum_{n=1}^{2m} \xx_{n}\right\|\ =\ 2m,\quad \mbox{ and } \quad \left\|\sum_{n=1}^{m} \xx_{2n}\right\|\ =\ m.
$$
It follows that
$$
\mathbf g_{m}\tilde{\boldsymbol{\mu}}_{m}\ \ge\ 2m^2. 
$$
On the other hand, since $\XB$ is normalized and $\left\Vert \xx_n^*\right\Vert\le 2$ for all $n$, we have $\max\left\lbrace \mathbf{g}_{m}, \la_{m}\right\rbrace\le 2m$. 
\end{remark}

It is known that there are semi-greedy fundamental minimal systems that are not even unconditional for constant coefficients \cite[Proof of Example 4.5]{BL2023a}. In such cases, $(\mathbf L_m^{ch})_{m\in \NN}$ is bounded but $( \mathbf g_m)_{m\in \NN}$ and $( \mathbf \la_m)_{m\in \NN}$ are not. On the other hand, for Markushevich bases in Banach spaces, we have $\max\{ \mathbf g_m,\la_m\}\lesssim \max \{(\mathbf L_m^{ch})^2, (\mathbf L_{m+1}^{ch})^2\}$ \cite[Propositions 6.5 and 6.11]{BL2023b}, and the proofs can be easily adapted to Schauder bases in $p$-Banach spaces. 

\begin{question}\normalfont  Can the estimate $\max\{ \mathbf g_m,\la_m\}\lesssim \max \{(\mathbf L_m^{ch})^2, (\mathbf L_{m+1}^{ch})^2\}$ be improved?  Specifically, is it possible to show that $\max\{ \mathbf g_m,\la_m\}\lesssim \mathbf L_m^{ch}$ for Markushevich or Schauder bases? 
\end{question}

Last but not least, we answer several questions about the optimal power in various bounds in \cite{AAB2}.
We define three \textit{ratio functions} $\mathfrak{R}_B, \mathfrak{R}_M$, and $\mathfrak{R}_S$ depending on $\XB$. 

\begin{definition}\normalfont
	Given two parameters $\mathbf a_m = \mathbf a_m[\XB, \mathbb X] > 0$ and $\mathbf b_m = \mathbf b_m[\XB, \mathbb X] > 0$, set
	
	\begin{align*}
	\mathfrak{R}_B(\mathbf a_m, \mathbf{b}_m)&\ := \ \inf U_B(\mathbf a_m, \mathbf b_m), \mbox{ where }\\
    U_B(\mathbf a_m, \mathbf b_m)&\ :=\ \{R: \mbox{ for every space }\XX\mbox{ and its fundamental system }\mathcal{X},\\
    &\mbox{ there exists }C(\XB, \XX)\mbox{ with }\sup_{m}\frac{\mathbf a_m}{(\mathbf b_m)^{R}} \ \le\  C(\XB, \XX)\},	
	\end{align*}

\begin{align*}
	\mathfrak{R}_M(\mathbf a_m, \mathbf b_m)&\ := \ \inf U_R(\mathbf a_m, \mathbf b_m), \mbox{ where }\\
    U_M(\mathbf a_m, \mathbf b_m)&\ :=\ \{R: \mbox{ for every space }\XX\mbox{ and its Markushevich basis }\mathcal{X},\\
    &\mbox{ there exists }C(\XB, \XX)\mbox{ with }\sup_{m}\frac{\mathbf a_m}{(\mathbf b_m)^{R}} \ \le\  C(\XB, \XX)\},\mbox{ and}
	\end{align*}
	
\begin{align*}
	\mathfrak{R}_S(\mathbf a_m, \mathbf b_m)&\ := \ \inf U_S(\mathbf a_m, \mathbf b_m), \mbox{ where }\\
    U_S(\mathbf a_m, \mathbf b_m) &\ :=\ \{R: \mbox{ for every space }\XX\mbox{ and its Schauder basis }\mathcal{X},\\
    &\mbox{ there exists }C(\XB, \XX)\mbox{ with }\sup_{m}\frac{\mathbf a_m}{(\mathbf b_m)^{R}} \ \le\  C(\XB, \XX)\}.
	\end{align*}
	\end{definition}

\begin{remark}\normalfont \label{remarkrelparam}
	 It follows from definitions that $\mathfrak{R}_S(\mathbf a_m, \mathbf b_m)\le \mathfrak{R}_M(\mathbf a_m, \mathbf b_m)\le \mathfrak{R}_B(\mathbf a_m, \mathbf b_m)$. 
\end{remark}

\begin{remark}\normalfont
    If $\mathbf b_m\ge 1$ and $r\notin U_B(\mathbf a_m, \mathbf b_m)$, then $\mathfrak{R}_B(\mathbf a_m, \mathbf b_m)\ge r$. Let $s < r$. For every $m\in \mathbb{N}$, 
    $$\frac{\mathbf a_m}{\mathbf b_m^r}\ \le\ \frac{\mathbf a_m}{\mathbf b_m^s}\ \Longrightarrow\ s\notin U_B(\mathbf a_m, \mathbf b_m).$$
    Hence, $\mathfrak{R}_B(\mathbf a_m, \mathbf b_m)\ge r$. 
\end{remark}

\begin{example}\normalfont
	From ${\boldsymbol{\mu}}_m \le ({\boldsymbol{\mu}}_m^d)^2$, $\tilde{\boldsymbol{\mu}}_m \le (\tilde{\boldsymbol{\mu}}_m^d)^2$, and \cite[Theorem 5.4]{BBGHO20}, we see that 
	$$\mathfrak{R}_M(\tilde{\boldsymbol{\mu}}_m, \tilde{\boldsymbol{\mu}}^d_m) \ =\ \mathfrak{R}_M({\boldsymbol{\mu}}_m, {\boldsymbol{\mu}}^d_m) \ =\ 2.$$
	Moreover, one can check that the proof works for all fundamental minimal systems, so
	$$\mathfrak{R}_B(\tilde{\boldsymbol{\mu}}_m, \tilde{\boldsymbol{\mu}}^d_m) \ =\ \mathfrak{R}_B({\boldsymbol{\mu}}_m, {\boldsymbol{\mu}}^d_m) \ =\ 2.$$
\end{example}

\begin{example}\label{exa1}\normalfont
	From \cite[Theorem 5.2]{BBGHO20} and \cite[Remark 5.3]{BBGHO20}, we know that $$\mathfrak{R}_S(\tilde{\boldsymbol{\mu}}_m, \tilde{\boldsymbol{\mu}}^d_m) \ =\ \mathfrak{R}_S({\boldsymbol{\mu}}_m, {\boldsymbol{\mu}}^d_m) \ \le 1.$$
	Let $\XB$ be a Schauder basis that is not democratic. For every $\varepsilon \in (0,1)$, since $\lim_{m\rightarrow \infty} \boldsymbol{\mu}_{m} = \infty$, it follows that
	$$\frac{\boldsymbol{\mu}_{m}}{(\boldsymbol{\mu}^d_{m})^\varepsilon}\ \ge\ \frac{\boldsymbol{\mu}_{m}}{(\boldsymbol{\mu}_{m})^\varepsilon}\ =\ \boldsymbol{\mu}_{m}^{1-\varepsilon}\rightarrow\infty.$$
	Therefore, $\varepsilon\notin U_S(\boldsymbol{\mu}_{m}, \boldsymbol{\mu}^d_{m})$, so $\mathfrak{R}_S(\boldsymbol{\mu}_m, \boldsymbol{\mu}^d_m)\ge \varepsilon$. Since $\varepsilon\in (0,1)$ is arbitrary, we must have $\mathfrak{R}_S(\boldsymbol{\mu}_m, \boldsymbol{\mu}^d_m) = 1$. Similarly, $\mathfrak{R}_S(\tilde{\boldsymbol{\mu}}_m, \tilde{\boldsymbol{\mu}}^d_m)  = 1$.
\end{example}

We are ready to state our answer to \cite[Question 8.3]{AAB2}: \textit{Given $0<p\le 1$, is there a constant $C$ such that $\la_m \le C\mathbf L_m^a$ for every basis of a $p$-Banach space?} The motivation for this question comes from a result on \cite[Page 23]{AAB2} that 
\begin{align}\label{des}
	\la_m\ \lesssim\ 
	\begin{cases}
		(\mathbf L_m^a)^2, & \mbox{ if } p=1;\\
		(\mathbf L_m^a)^{\left(2+\frac{1}{p}\right)}, &\mbox{ if } p<1.
	\end{cases}
\end{align}
The authors in \cite{AAB2} asked whether the power of $\mathbf L_m^a$ in these bounds can be reduced to $1$; they conjectured that such an improvement is probably impossible. Our next result answers their question. 

\begin{theorem}\label{main3}
    Let $\mathbb{X}$ be a $p$-Banach space with a fundamental minimal system $\XB$. There is a constant $C = C(p)$ with $\la_m\le C (\mathbf L^a_m)^2$. Furthermore, 
    $$\mathfrak{R}_M(\la_{m}, \mathbf L_m^a)\ =\ \mathfrak{R}_B(\la_{m}, \mathbf L_m^a)\ =\ 2\mbox{ and } \mathfrak{R}_S(\la_m, \mathbf L_m^a)\ =\ 1.$$
\end{theorem}

\begin{remark}\normalfont
   Theorem \ref{main3} improves the bound for $p<1$ in \eqref{des} by replacing $(2+1/p)$ with $2$. Moreover, the theorem confirms that the power $2$ in $\la_m\lesssim \mathbf (\mathbf L_m^a)^2$ is optimal for fundamental minimal systems and for Markushevich bases.
\end{remark}

We also answer \cite[Question 8.4]{AAB2}: 
	\textit{Thanks \cite[Lemma 4.1]{AAB2} that $\la_m\lesssim (\la_m^d)^2$ and \cite[Inequality (8.6)]{AAB2}, we obtain $\boldsymbol{\nu}_m\lesssim (\la_m^d)^2$, $\tilde{\boldsymbol{\mu}}_m\lesssim(\la_m^d)^2$ and $\boldsymbol{\mu}_m \lesssim (\la_m^d)^2$. Can any of these asymptotic estimates be improved?} Here the parameter $\boldsymbol{\nu}_m$, which is associated with the symmetry for largest coefficients \cite{AAB2,BBG2017}, is defined as $$\boldsymbol{\nu}_m \ =\ \sup\frac{\| \Ind_{\varepsilon, A}+f\|}{\| \Ind_{\delta, B}+f\|},$$
where the supremum is taken over all finite subsets $| A|=| B|\le m$, all signs $\varepsilon\in \mathcal E_A, \delta\in \mathcal E_B$, all $f\in\mathbb X$ with $\max_{n}|\xx_n^*(f)|\le 1$ and $\supp(f)\cap (A\cup B)=\emptyset$.  Under the additional assumption $A\cap B=\emptyset$, we have the parameter $\boldsymbol\nu_m^d$. 
Symmetry for large coefficients is the key property to characterize greedy and almost greedy bases with constant $1$ and was studied in \cite{AW, AA2017, AABB, AABCO}, while its associated parameters $\boldsymbol\nu_m$ and $\boldsymbol\nu_m^d$ were investigated in \cite{AAB2,BBG2017,BBGHO2018}. 

\begin{thm}\label{m4}
    We have $\mathfrak{R}_S(\la_m, \la_m^d) = 1$, while $\mathfrak{R}_B(\la_m, \la_m^d) = \mathfrak{R}_M(\la_m, \la_m^d) = 2$. As a result,  
    \begin{itemize}
        \item $\mathfrak{R}_B(\boldsymbol\mu_{m}, \la^{d}_m) = \mathfrak{R}_M(\boldsymbol\mu_{m}, \la^{d}_m) = 2$,
        \item $\mathfrak{R}_B(\tilde{\boldsymbol\mu}_{m}, \la^{d}_m) = \mathfrak{R}_M(\tilde{\boldsymbol\mu}_{m}, \la^{d}_m) = 2$, and
        \item $\mathfrak{R}_B(\boldsymbol{\nu}_{m},\la_m^d)  =  \mathfrak{R}_M(\boldsymbol{\nu}_{m},\la_m^d) = 2$.
    \end{itemize}
\end{thm}

\subsection{Further notation}
%Let $\mathbb{E} = \{\varepsilon = \{\varepsilon_n\}_n\subset \mathbb{K}: |\varepsilon| = 1\}$. 

Let $G(f, m)$ denote the set of all greedy sets of $f$ of order $m$ and let $G(f)$ be the set of all greedy sets of $f$; that is, $G(f) = \cup_{m\ge 0}G(f,m)$.

Given a number $a$, define 
$$\mbox{sign}(a) \ =\ \begin{cases} \frac{a}{|a|}, &\mbox{ if } a \not= 0;\\ 1, &\mbox{ if } a=0.\end{cases}$$ 

When $\XB$ is Schauder, we define the basis constant $K_b: = \sup_{n} \|P_n\|$.

For $n\in \NN$, let $I_n$ and $I_n^c$ denote the interval $\{1,\dots, n\}$ and the set $\NN\backslash I_n$, respectively and let $I_0 := \emptyset$. For $a\in \FF$ and $A\subset \FF$, let $a+A:=\{ a+x: x\in A\}$ and $aA:=\{ ax: x\in A\}$. If $a\in \mathbb{R}$ and $A\subset \RR$, let $A_{>a}:=A\cap (a,\infty)$. We use similar notation for $\ge$, $<$ and $\le$.

Given a set $A$, $A^{<\infty}$ denotes the collection of all finite subsets of $A$, whereas for each $m\in\NN_0$, $A^{(m)}$ consists of all subsets of $A$ with exactly $m$ elements, and $A^{\le m} :=\bigcup_{k=0}^{m}A^{(k)}$. 

Finally, for every sequence $(\textbf{a}_m)_{m\in \NN}$ of some parameter $\textbf a_m$, we assume that $\textbf{a}_0=0$.

\subsection{Structure of the paper}
The paper is organized as follows: Section \ref{S2} proves our first two results, Theorem \ref{main1} and Theorem \ref{cheby}; Section \ref{S4} records a technical result that helps verify Markushevich bases; Section \ref{S5} proves Theorem \ref{m4}, thus answering \cite[Question 8.4]{AAB2}; Section \ref{S6} studies the ratios of $\boldsymbol{\mu}_m$ and $\mathbf{L}_m^a$, proves Theorem \ref{main3}, and computes several ratios that involve $\boldsymbol\nu_m$.

\section{Proofs of Theorems \ref{main1} and \ref{cheby}}\label{S2}
We need the following lemma about $p$-convexity. Set $\mathbf A_p = (2^p-1)^{1/p}$.

\begin{lemma}\cite[Corollary 1.3]{AABW2021}\label{convexity}
Let $\mathbb X$ be a $p$-Banach space for some $0<p\le 1$ and let $\mathcal B=(\xx_n)_n$ be a fundamental minimal system for $\XX$. For all scalars $(a_j)_{j\in A}$ with $A$ a finite set and $| a_j|\le 1$ for every $j\in A$,
$$\left\| \sum_{n\in A}a_n \xx_n\right\|\ \le\ \mathbf A_p \sup\{\|\Ind_{\varepsilon,A}\| : \varepsilon\in\mathcal E_A\}.$$
\end{lemma}

\begin{proof}[Proof of Theorem \ref{main1}]
First, we prove that $C_1\mathbf L_m^s\ \le\ \max\{\la_m^c, \mathbf g_m\}$ for some constant $C_1 = C_1(p)$.
Let $f\in\mathbb X$, $A$ greedy set of $f$ with $|A| \le m$, $k\le |A|$, and $D:=\{ 1,\dots, k\}$. We decompose
$$f-P_A(f)\ =\ ((f-P_k(f))-P_{A\backslash D}(f-P_k(f)))+P_{D\backslash A}(f).$$

Since $A\backslash D$ is a greedy set of $f-P_k(f)$ and $|A\backslash D|\le m$,
\begin{equation}\label{four}
	\| (f-P_k(f))-P_{A\backslash D}(f-P_k(f))\| \ \le\  \mathbf g_m^c\| f-P_k(f)\|.
\end{equation}

Moreover, by Lemma \ref{convexity},

\begin{align}\label{five}
	\|P_{D\backslash A}(f)\| &\ \le \mathbf A_p\max_{n\in D\backslash A} |\xx_n^*(f)|\sup\{\|\Ind_{\varepsilon,D\backslash A}\| : \varepsilon\in\mathcal E_{D\backslash A}\} \nonumber\\
	&\ \le\ \mathbf A_p\min_{n\in A\backslash D}| \xx_n^*(f-P_k(f))|\sup\{\|\Ind_{\varepsilon,D\backslash A}\| : \varepsilon\in\mathcal E_{D\backslash A}\} \nonumber\\
	&\ \le\ \mathbf A_p\la_m^c\|f-P_k(f)\|,
\end{align}
where the last inequality is due to the fact that $A\backslash D$ is a greedy set of $f-P_k(f)$, $D\backslash A\le m$, $D\backslash A < \supp (f-P_k(f))$, and $|D\backslash A|\le |A\backslash D|\le m$.

By \eqref{four} and \eqref{five}, 
\begin{align*}\| f-P_A(f)\|^p &\ \le\  ((\mathbf g_m^c)^p+ \mathbf A_p^p(\la_m^c)^p)\| f-P_k(f)\|^p\\
&\ \le\ (1+\mathbf g_m^p + \mathbf A_p^p(\la_m^c)^p)\| f-P_k(f)\|^p\\
&\ \le\ (2\mathbf g_m^p + \mathbf A_p^p(\la_m^c)^p)\| f-P_k(f)\|^p;
\end{align*}
hence, $$\mathbf L_m^s\ \le\ (2+\mathbf A_p^p)^{1/p}\max\{ \mathbf g_m,\la_m^c\}.$$

We now prove that $\max\{\la_m^c, \mathbf g_m\}\  \le\ C_2\mathbf L_m^s$ for some $C_2 = C_2(p)$.
For $f\in\mathbb X$ and a greedy set $A$ with  $|A|\le m$, it follows from \eqref{partially} that
$$
	\| f-P_A(f)\| \ \le\ \mathbf L_m^s\| f\|,
$$
thus, 
\begin{equation}\label{three}\mathbf g_m^c\ \le\ \mathbf L_m^s.\end{equation} 
It follows that
\begin{eqnarray}\label{two}
	\mathbf g_m^p \ \le\ 1+(\mathbf g_m^c)^p \ \le\ 2(\mathbf L_m^s)^p,
\end{eqnarray}
so
$$\mathbf g_m \ \le\ 2^{1/p} \mathbf L_m^s.$$

Next, take $f\in\mathbb X$, $B$ a greedy set of $f$, $A\subseteq\mathbb N$ such that $A \le m$, $A< \supp (f)$, $|A|\le |B|\le m$, and $\varepsilon\in\mathcal E_A$. Choose $U \subset [1, \max A]\backslash A$ so that $\max A\le |B\cup U|\le m$. Set $t:=\min_{n\in B}|x_n^*(f)|$ and define $g:=f+t1_{\varepsilon, A} + t1_{U}$. Since $B\cup U$ is a greedy set of $g$, we have
\begin{align*}
	(t\|1_{\varepsilon, A}\|)^p&\ \le\ \| g-P_{B\cup U}(g)\|^p+\| f-P_B(f)\|^p\\
	&\ \le\ (\mathbf L_{m}^s\| g-P_{\max A}(g)\|)^p+ (\mathbf g_m^c)^p\| f\|^p\\
	&\ =\ ((\mathbf L_m^s)^p + (\mathbf g_m^c)^p)\| f\|^p\\
	&\ \le\  2(\mathbf L_m^s)^p \| f\|^p, \quad \mbox{ due to \eqref{three}}.
\end{align*}
Hence, 
$$\la_m^c\ \le\ 2^{1/p}\mathbf L_m^s.$$
This completes our proof. 
\end{proof}

To prove Theorem \ref{cheby}, we recall the \textit{truncation operator} $T_\alpha$ introduced in \cite{DKK2003}. For $\alpha>0$ and $y\in\mathbb F$,
$$
T_\alpha(y) \ :=\
\begin{cases}
	\alpha\sgn(y), & \mbox{ if } | y| > \alpha; \\
	y, & \mbox{ if } | y| \le \alpha.
\end{cases}
$$

Let $\varepsilon(f) := (\sgn(\xx_n^*(f)))_{n=1}^\infty$.
With an abuse of notation, for $f\in\XX$,
$$T_\alpha(f)\ :=\ \alpha\Ind_{\varepsilon(f), G_\alpha(f)}+(f-P_{G_\alpha(f)}(f)),$$
where $G_\alpha(f):=\{ n : | \xx_n^*(f)|>\alpha\}$. Since $G_\alpha(f)$ is a finite set, $T_\alpha(f)$ is well defined.

Define 
$$\mathbf r_m\ :=\ \sup\left\{ \frac{\min_{n\in A}|\xx_n^*(f)|\| \Ind_{\varepsilon(f),A}\|}{\| f\|}\, :\, f\in \XX\backslash \{0\}, A\mbox{ greedy set of }f,| A|\le m\right\}$$
and for $u > 0$, 
$$\eta_p(u)\ :=\ \min_{0 < t< 1}(1-t^p)^{-1/p}(1-(1+\mathbf A_p^{-1}u^{-1}t)^{-p})^{-1/p}.$$
The function first appeared in \cite[Theorem 4.8]{AABW2021}. As discussed in \cite[Remark 4.9]{AABW2021}, 
$$\eta_p(u)\ \approx\ u^{1/p},\quad u\ge 1.$$

For the proof of the following auxiliary result, we refer the readers to \cite[Theorem 4.8]{AABW2021} and the discussion on \cite[Page 14]{AAB2}.

\begin{theorem}\label{trunc}
Let $\mathcal X$ be a fundamental minimal system for a $p$-Banach space. Then
$$\mathbf r_m\ \le\ \mathbf g_m \eta_p(\mathbf g_m), \quad m\in \mathbb{N}.$$
\end{theorem}

\begin{proof}[Proof of Theorem \ref{cheby}] Take $f\in\XX$, $A=A_m(f)$, and $y\in \spn(\XB)$ with $| \supp(y)|\le m$. Let $\alpha:=\max_{n\not\in A}|\xx_n^*(f)|$ and let
	$$g\ :=\ P_A(f)-P_{A}(T_\alpha(f-y)).$$
    We have
    \begin{align*}
        g\ =\ P_A(f-T_\alpha(f-y))&\ =\ (I-P_{\supp(y)\backslash A})(f-T_\alpha(f-y))\\
        & \ = \ f-T_\alpha(f-y) - P_{\supp(y)\backslash A}(f-T_\alpha(f-y)).
    \end{align*}
    Hence, 
    $$f-g\ =\ P_{\supp(y)\backslash A}(f-T_{\alpha}(f-y))+T_\alpha(f-y).$$

By Theorem \ref{trunc} and the fact that $G_{\alpha}(f-y)\subset A\cup\supp(y)$, 
	\begin{align*}
	\|T_\alpha(f-y)\|^p &\ \le\ \alpha^p\|\Ind_{\varepsilon(f-y),G_{\alpha}(f-y)}\|^p + \| P_{(G_{\alpha}(f-y))^c}(f-y)\|^p\\
    &\ \le\ \mathbf r^p_{2m}\|f-y\| + \mathbf (\g_{2m}^c)^p\|f-y\|\\ 
    &\ \le \ (\mathbf g^p_{2m}\eta^p_p(\mathbf g_{2m})+\mathbf (\g_{2m}^c)^p)\| f-y\|\\
    &\ \lesssim\ \mathbf g^{1+p}_{m} \|f-y\|.
\end{align*}

To bound $P_{\supp(y)\backslash A}(f-T_{\alpha}(f-y))$, choose a greedy set $D$ of $f-y$ with $|D| = |\supp(y)\backslash A|\le |A\backslash \supp (y)|$. Since $\min_{n\in A\backslash \supp(y)} |\xx^*_n(f-y)|\ge \alpha$, it follows that 
$$\min_{n\in D}|\xx^*_n(f-y)|\ \ge\ \alpha.$$
We have
\begin{align*}
\| P_{\supp (y)\backslash A} (f-T_\alpha(f-y)) \| &\ \le \ 2\mathbf A_p\alpha\sup_{|\eta|=1}\|\Ind_{\eta,\supp (y)\backslash A}\|\\
&\ \le\ 2\mathbf A_p\min_{n\in D}| \xx_n^*(f-y)|\sup_{|\eta|=1}\|\Ind_{\eta,\supp (y)\backslash A}\|\\
&\ \le\ 2\mathbf A_p \la_m \| f-y\|,
\end{align*}
where the last inequality is due to the fact that  $D$ is a greedy set of $f-y$ and 
$|D| = |\supp(y)\backslash A| \le m$.

This completes our proof that $\mathbf L_m^{ch}\lesssim \max\{\mathbf g^{1+1/p}_{m}, \la_m\}$.

If $\XX$ is a Banach space, the same proof works to show that
$$\mathbf L_m^{ch}\ \lesssim\ \max\{\mathbf g_{m},\la_m\}.$$
\end{proof}

\section{A technical result for Markushevich bases}\label{S4}

The following lemma will be used in due course.
\begin{lemma}\label{lemmaMarkushevich} Let $I$ and $J$ be nonempty sets and let $(\|\cdot\|_i)_{i\in I}$ and $(\|\cdot\|_j)_{j\in J}$ be families of semi-norms on $\mathtt{c}_{00}$ with the following properties: 
\begin{enumerate}
\item   for each $i\in I$, $\|\cdot\|_i$ is finitely supported; that is, there is a finite set $A_i\subset \NN$ such that if $a_n=0$ for each $n\in A_i$, then 
\begin{equation}\label{fsup}\|(a_n)_{n\in\NN}\|_i\ =\ 0;\end{equation}
\item there exists $K\ge 1$ such that if $A<B$ are finite subsets of $\NN$, $ (a_n)_{n\in\NN}$ is supported in $A$, and $(b_n)_{n\in\NN}$ is supported in $B$, then 
$$
\|(a_n)_{n\in\NN}\|_j\ \le\ K \| (a_n)_{n\in\NN}+(b_n)_{n\in\NN}\|_j,\quad j\in J. 
$$
\end{enumerate}
Define on $\mathtt{c}_{00}$ the semi-norm
$$
\|(a_n)_{n\in\NN}\|\ :=\ \max\{ \sup_{i\in I}\|(a_n)_{n\in\NN}\|_i , \sup_{j\in J}\| (a_n)_{n\in\NN}\|_j\}.
$$
Suppose that $\|\cdot\|$ is a norm. Let $\XX$ be the completion of $\mathtt{c}_{00}$ with respect to $\|\cdot\|$. If $\XB=(\xx_n)_{n\in\NN}$ is the canonical vector system of $\XX$, then $\XB$ is a Markushevich basis.
\end{lemma}
\begin{proof}
Suppose the statement is false. Pick $x\in \XX$ so that $\|\xx_n\| = 1$, $\xx_n^*(x)=0$ for all $n\in\NN$, and a sequence $(x_n)_{n\in \NN}\subset \spn(\XB)$ with $\|x_n\| = 1$ and $\lim_{n\rightarrow \infty}\|x_n-x\| = 0$. Let $y_1:=x_1$. Since the $x_k$'s are in the span of $\XB$ and for each $n\in \mathbb{N}$, $\lim_{k\rightarrow\infty} \xx_n^*(x_k) = 0$,  there is $k_2>1$ such that 
$$\| P_{\{ 1, \dots, \max(\supp(y_1))\}}(x_{k_2})\|\ <\ \frac{1}{2^1}.$$ 
Let 
$$y_2\ :=\ \frac{x_{k_2}-P_{\{ 1, \dots, \max(\supp(y_1))\}}(x_{k_2})}{\| x_{k_2}-P_{\{ 1, \dots, \max(\supp(y_1))\}}(x_{k_2})\|}.$$ 
In general, supposing that $y_m$ has been constructed, choose $k_{m+1} > k_m$ such that
$$\| P_{\{ 1, \dots, \max(\supp(y_m))\}}(x_{k_{m+1}})\|\ <\ \frac{1}{2^m}.$$ 

We obtain a sequence $(y_n)_{n\in \NN}\subset \spn(\XB)$ with $\|y_n\| = 1$, $\lim_{n\rightarrow\infty}\|y_n-x\| = 0$, and $\supp(y_n)<\supp(y_{n+1})$ for each $n\in \NN$. Passing to a further subsequence, we may assume that $\| x-y_n\|\le 1/(4K)$ for all $n\in \NN$. Therefore, for each $j\in J$ and $n\in \NN$, we have
$$
\| y_n\|_j \ \le\ K \| y_n-y_{n+1}\|_j\ \le\  K \| x-y_{n+1}\|_j+K \| y_n-x\|_j\ \le\ \frac{1}{2}, 
$$
so there is $i_n\in I$ such that 
$$
\| y_n\|_{i_n}\ >\ 1-\frac{1}{2^n}.
$$
Using \eqref{fsup} and passing to a new subsequence $(y_m)_{m=1}^\infty$, we may assume that for each $n<m$, $\supp(y_m)>A_{i_n}$. Thus, for $n\in \NN$, we have
$$
1-\frac{1}{2^n}\ <\ \| y_n\|_{i_n}\ \le \ \| y_n-y_{n+1}\|_{i_n}\ \le\ \| y_n-x\|+\| x-y_{n+1}\|\ \le\ \frac{1}{2K}\ \le\ \frac{1}{2},
$$
a contradiction.
\end{proof}

\section{\texorpdfstring{The ratios of \(\boldsymbol{\lambda}_m\) and \(\la^d_m\)}{The ratio of lambda parameters}}\label{S5}

In this section, we compute the ratio functions between $\la_m$ and $\la_m^d$. We prove that for Schauder bases,
$\mathfrak{R}_S(\la_m, \la_m^d) = 1$, while for Markushevich bases and fundamental minimal systems, the corresponding ratios equal $2$.

\subsection{\texorpdfstring{Value of \(\mathfrak{R}_S(\la_{m}, \la^{d}_m)\)}{Value of R_S(lambda_m, lambda^d_m)}}

\begin{theorem}\label{t3v3}If $\XB$ is a Schauder basis with basis constant $K_b$ for a $p$-Banach space $\XX$, there exists an absolute constant $C = C(\XB, \XX)$ such that $\la_m \le C\la^d_m$. In particular, 
$$
\la_m\ \le\ ((\la_m^d)^p (2K_b^p+1)+2K_b^p )^{\frac{1}{p}}, \quad m\in \NN. 
$$
\end{theorem}
\begin{proof}
Fix $f\in \XX$ with finite support, $m\in \NN$, $A\subset \NN$ with $|A|\le m$, $\varepsilon\in \EE_A$, and $B\in G(f,| A|)$. Let $\alpha:=\min_{n\in B}| \xx_n^*(f)|$. If $|A|=1$, and there is $n\in \NN$ such that $A=B=\{n\}$, then
\begin{align*}\alpha^p\|\Ind_{\varepsilon, A}\|^p\ =\ \| \xx_n^*(f)\xx_n\|^p&\ \le\ \| P_{\{1, \ldots, n\}}(f)\|^p + \|P_{\{1, \ldots, n-1\}}(f)\|^p\\
&\ \le\ K_b^p \|f\|^p + K_b^p\|f\|^p\ =\ 2K_b^p \|f\|^p,
\end{align*}
so $$\alpha\|\Ind_{\varepsilon, A}\|\ \le\ 2^{1/p} K_b \|f\|.$$

If $|A|=1$ and $A\cap B=\emptyset$, let $B=\{n\}$. Since $B\in G(f, 1)$, we have 
$$
\alpha\|\Ind_{\varepsilon, A}\|\ =\ | \xx_n^*(f)| \|\Ind_{\varepsilon, A}\|\ \le\ \la_{1}^d\| P_{\{n\}}(f)\|\ \le\ 2^{\frac{1}{p}}K_b\la_{m}^d\|f\|. 
$$

If $|A|>1$, let 
$$
n_1\ :=\ \min\{ {n\in \NN}\,:\, |I_n\cap B|\ge |I_n^c\cap A|\}
$$
and define 
\begin{align*}
l_1   &\ :=\ | I_{n_1}\cap B|, 
&\quad r_1   &\ :=\ | I_{n_1}^c\cap A|;\\
x_{L} &\ :=\ P_{I_{n_1}}(f), 
&\quad x_{R} &\ :=\ f - P_{I_{n_1}}(f);\\
A_{L} &\ :=\ A\cap I_{n_1}, 
&\quad A_{R} &\ :=\ A\backslash A_{L};\\
B_{L} &\ :=\ B\cap I_{n_1}, 
&\quad B_{R} &\ :=\ B\backslash B_{L}.
\end{align*}
Note that $B_{L}\in G(x_{L})$ and $B_{R}\in G(x_{R})$. Also, it follows from definitions that $r_1\le l_1\le r_1+1$. We proceed by case analysis. 

Case 1: If $l_1=r_1$, then $|A_{R}|=r_1=l_1=|B_{L}|$, so $|A_{L}|=|A|-r_1=|B_{R}|$. Hence, 
$$
\min_{n\in B_{L}}| \xx_n^*(f)| \|\Ind_{\varepsilon, A_{R}}\| \ \le\ \la_{l_1}^d\| x_{L}\|\ \le\ \la_{m}^d K_b\| f\|
$$
and 
$$
\min_{n\in B_{R}}| \xx_n^*(f)|\|\Ind_{\varepsilon, A_{L}}\| \ \le\ \la_{|A|-r_1}^d\| x_{R}\|\ \le\ \la_{m}^d (1+K_b^p)^{\frac{1}{p}}\| f\|. 
$$
Thus, by the $p$-triangle inequality, we obtain 
\begin{align*}
\alpha \|\Ind_{\varepsilon, A}\|&\ \le\ \left(\min_{n\in B_{L}}| \xx_n^*(f)|^p \|\Ind_{\varepsilon, A_{R}}\|^p+ \min_{n\in B_{R}}| \xx_n^*(f)|^p\|\Ind_{\varepsilon, A_{L}}\|^p\right)^{\frac{1}{p}}\\
&\ \le\ \la_{m}^d(1+2K_b^p)^{\frac{1}{p}}\| f\|.
\end{align*}

Case 2: If $r_1=l_1-1$, then there is a nonnegative $n_2<n_1$ such that $|B\cap I_{n_2}|=|A_R|$. Since $B\cap I_{n_2}\in G(P_{I_{n_2}}(x))$, we have
$$
\alpha \|\Ind_{\varepsilon, A_{R}}\|\ \le\ \la_{m}^d \| P_{I_{n_2}}(f)\|\ \le\ \la_{m}^d K_b\| f\|.
$$
The minimality of $n_1$ and the fact that $l_1>r_1$ together imply that $n_1\in A\cap B$. Hence, 
$$
\| \alpha \varepsilon_{n_1}\xx_{n_1}\|\ \le\ \|\xx_{n_1}^*(f)\xx_{n_1}\|\ \le\ 2^{\frac{1}{p}}K_b \| f\|. 
$$
Let $A_{L,L}:=A_{L}\backslash \{ n_1\}$. 
Then 
$$|A_{L, L}| \ =\ |A_L| - 1 \ =\ |A| - (r_1 + 1) \ =\ |B| - l_1 \ =\ |B_R|.$$
Hence,
$$
\alpha  \|\Ind_{\varepsilon, A_{L,L}}\|\ \le\ \la_m^d\| x_{R}\| \ \le\ \la_m^d(1+K_b^p)^{\frac{1}{p}}\| f\|. 
$$
From the above estimates and the $p$-triangle inequality, we obtain 
$$
\alpha  \|\Ind_{\varepsilon, A}\|\ \le\  ((\la_m^d)^p (2K_b^p+1)+2K_b^p)^{\frac{1}{p}}\|f\|.
$$ 
\end{proof}

\begin{corollary}\label{corolario?}
We have 
$$\mathfrak{R}_S(\la_{m}, \la^{d}_m)\ =\ 1.$$
\end{corollary}

\begin{proof}
Theorem \ref{t3v3} implies that $\mathfrak{R}_S(\la_{m}, \la^{d}_m)\le 1$. To see that $\mathfrak{R}_S(\la_{m}, \la^{d}_m)\ge 1$, we choose a Schauder basis which is not democratic. Then $\lim_{m\rightarrow\infty} \boldsymbol{\mu}_m = \infty$. Since $\boldsymbol\mu_m \le \la_m$, it follows that $\lim_{m\rightarrow\infty}\la_m = \infty$. Use the same reasoning as in Example \ref{exa1} to complete the proof. 
\end{proof}

\begin{remark}\label{remarkdemoc}\normalfont The proof of Theorem \ref{t3v3} with only minor modifications can be used to estimate $\boldsymbol\mu_m$ and $\tilde{\boldsymbol\mu}_m$ in terms of $K_b$ and $\boldsymbol\mu_m^d$ and $\tilde{\boldsymbol\mu}^d_m$ respectively, and the same estimate holds in this case. This extends \cite[Theorem 5.2 and Remark 5.3]{BBGHO20} to $p$-Banach spaces, with a different estimate for the case $p=1$. 
\end{remark}

\subsection{\texorpdfstring{Value of \(\mathfrak{R}_B(\la_{m}, \la^{d}_m)\) and \(\mathfrak{R}_M(\la_{m}, \la^{d}_m)\)}{Value of R_B and R_M}}\label{subsectionR_Blambda}

We compute  
\begin{itemize}
    \item $\mathfrak{R}_B(\la_{m}, \la^{d}_m)$ and $\mathfrak{R}_M(\la_{m}, \la^{d}_m)$;
    \item $\mathfrak{R}_B(\boldsymbol\mu_{m}, \la^{d}_m)$ and $\mathfrak{R}_M(\boldsymbol\mu_{m}, \la^{d}_m)$;
    \item $\mathfrak{R}_B(\tilde{\boldsymbol\mu}_{m}, \la^{d}_m)$ and $\mathfrak{R}_M(\tilde{\boldsymbol\mu}_{m}, \la^{d}_m)$;
    \item $\mathfrak{R}_B(\boldsymbol{\nu}_{m},\la_m^d)$ and $\mathfrak{R}_M(\boldsymbol{\nu}_{m},\la_m^d)$,
\end{itemize}
 giving a negative answer to \cite[Question 8.4]{AAB2}. 

\begin{proposition}\label{propositionmulambda}There is a space $\XX$ with a Markushevich basis $\XB$ for which 
\begin{equation}
\limsup_{m\to \infty}\frac{\boldsymbol\mu_{m}}{(\la_{m}^d)^{2-\varepsilon}}\ =\ \infty.\label{toprovemumlabdaalt}
\end{equation}
\end{proposition}
\begin{proof}
Choose two sequences $(n_k)_{k=1}^\infty\subset 2\NN$ and $(m_k)_{k = 1}^\infty\subset \NN_{>5}$ such that, for each $k\in \NN$,
$$
2m_k^{4k}\ <\ n_k\ <\ \frac{m_{k+1}}{2}.
$$
Set $n_0=m_0 =0$. We define on $\mathtt{c}_{00}$
$$
D_k((a_n)_{n\in\NN})\ :=\ \frac{1}{m_k}\sup_{\substack{A\subset \NN\\ A>2n_k\\ 1\le |A|\le n_k}}\frac{1}{\sqrt{|A|} }\sum_{n\in A}| a_n|, \quad k\in \mathbb{N}. 
$$

Let $\Delta_k$ be the set of finite sequences $\delta=(\delta_j)_{1\le j\le 2n_k}$ with the following properties: 
\begin{enumerate}[\rm 1.]
\item $\left|\sum_{j=1}^{2n_k}\delta_j \right|\le 1$. 
\item For $1\le j\le 2n_k$, $ | \delta_j|\le 1$. 
\item For each $A\subset I_{n_k}$, $\sum_{j\in A}| \delta_j| \le \sqrt{|A|}$. 
\end{enumerate}
Define on $\mathtt{c}_{00}$  
$$
Q_k((a_n)_{n\in\NN})\ :=\ \frac{1}{m_k}\left| \sup_{\delta \in \Delta_k}\sum_{n=1}^{2n_k}\delta_n  a_n \right|, \quad k\in \mathbb{N}.
$$

Let $\XX$ be the completion of $\mathtt{c}_{00}$ with the norm
$$
\| (a_n)_{n\in\NN}\|\ :=\ \sup_{k\in \NN}\{ | a_k|, D_k((a_n)_{n\in\NN}), Q_k((a_n)_{n\in\NN})\}, 
$$
and let $\XB$ be the canonical unit vector system of $\XX$.  Then $\XB$ and $\XB^*$ are both normalized, and that $\XB$ is a Markushevich basis by Lemma \ref{lemmaMarkushevich}. 

To see that \eqref{toprovemumlabdaalt} holds, first we claim that
\begin{equation}
Q_k(x)\ =\ \frac{1}{m_k}\sum_{n\in \{n_k+1, \ldots, 2n_k\}}| \xx_n^*(x)|, 
\label{claim1mu}
\end{equation}
for all $$x\in \spn\{ \xx_n: n\ge n_k + 1\}\mbox{ with } |\{n_k + 1, \ldots, 2n_k\}\cap \supp(x)|\ \le\ \frac{n_k}{2};$$
furthermore,
\begin{equation}
Q_k(x)\ \ge\ \frac{1}{m_k\sqrt{n_k}}\sum_{n\in \NN}| \xx_n^*(x)|,\quad x\in \spn\{ \xx_n: n\in I_{n_k}\}.\label{claim2mu}
\end{equation}
To prove \eqref{claim1mu}, first note that the inequality
$$
Q_k(x)\ \le \ \frac{1}{m_k}\sum_{n\in \{n_k+1, \ldots, 2n_k\}}| \xx_n^*(x)|
$$
is immediate from the definition of $Q_k$. Now pick $B\subset \{n_k+1, \ldots, 2n_k\}\backslash \supp(x)$ with $|B|= | \{n_k+1, \ldots, 2n_k\}\cap \supp(x)|$. Let $\sigma:B\rightarrow \{n_k+1, \ldots, 2n_k\}\cap\supp(x)$ be a bijection and define for each $1\le n\le 2n_k$,
$$
\delta_{n}\ :=\ 
\begin{cases}
\frac{|\xx_n^*(x)|}{\xx_n^*(x)}, &\mbox{ if } n\in \supp(x); \\
-\frac{ |\xx_{\sigma(n)}^*(x)|}{\xx_{\sigma(n)}^*(x)},  &\mbox{ if } n\in B;\\
0, &\mbox{ otherwise}. 
\end{cases}
$$
It is routine to check that $\delta=(\delta_n)_{1\le n\le 2n_k}\in \Delta_k$ and that
$$
\sum_{n=1}^{2n_k}\delta_n \xx_n^*(x)\ =\ \sum_{n\in \{n_k+1, \ldots, 2n_k\}}| \xx_n^*(x)|, 
$$
which completes the proof of \eqref{claim1mu}. The proof of \eqref{claim2mu} is similar: for $1\le n\le 2n_k$, we define
$$
\delta_{n}\ :=\
\begin{cases}
\frac{|\xx_n^*(x)|}{\sqrt{n_k}\xx_n^*(x)}, &\mbox{ if } n\in \supp(x); \\
-\frac{|\xx_{n-n_k}^*(x)|}{\sqrt{n_k}\xx_{n-n_k}^*(x)}, &\mbox{ if } n-n_k\in \supp(x);\\
0, & \mbox{ otherwise},
\end{cases}
$$
and \eqref{claim2mu} follows.

Now fix $k\ge 2$. To find a lower bound for $\boldsymbol\mu_{n_k}$, we estimate $\|\Ind_{\{1,\ldots, 2n_k\}}\|$. We have 
$$
Q_k(\Ind_{\{1, \ldots, 2n_k\}})\ \le\ \frac{1}{m_k} \quad\mbox{ and }\quad D_k(\Ind_{\{1, \ldots, 2n_k\}})\ =\ 0. 
$$
For $1\le l<k$, 
$$
\max\{D_l(\Ind_{\{1,\ldots, 2n_k\}}), Q_l(\Ind_{\{1,\ldots, 2n_k\}})\}\ \le\ 2n_{k-1}\ \le\ m_k, 
$$
whereas for $l>k$, 
$$
Q_l(\Ind_{\{1,\ldots, 2n_k\}})\ \le\ \frac{2n_k}{m_{l}}\ \le\  1 \quad \mbox{ and }\quad D_l(\Ind_{\{1,\ldots, 2n_k\}})\ =\ 0. 
$$
Now let $A:=\{n_k+1, \ldots, 3n_k/2\}\cup \{2n_k+1,\ldots, 7n_k/2\}$. By \eqref{claim1mu}, we have
$$
\| \Ind_{A}\|\ \ge\ Q_k(\Ind_{A})\ =\ \frac{n_k}{2m_k}. 
$$
Therefore,
\begin{equation}
\boldsymbol\mu_{2n_k}\ \ge\ \frac{\|\Ind_A\|}{\|\Ind_{\{1,\ldots, 2n_k\}}\|}\ \ge\ \frac{n_k}{2m_k^2}. \label{mu1}
\end{equation}

Next, we find an upper bound for $\la_{n_k}^d$. To this end, fix $x\in \XX$, $1\le m\le n_k$, $B\in G(x,m)$, $A\in \NN^{(m)}$ with $\supp(x)\cap A=\emptyset$, and $\varepsilon\in \EE_A$. We may assume by scaling that $\min_{n\in B}|\xx_n^*(x)|=1$. For each $l>k$, we have
$$
\max\{ Q_l(\Ind_{\varepsilon, A}), D_l(\Ind_{\varepsilon, A})\}\ \le\ \frac{n_k}{m_l}\ \le\ 1\ \le\ \| x\|. 
$$
On the other hand, for $1\le l<k$, 
$$
\max\{ Q_l(\Ind_{\varepsilon, A}), D_l(\Ind_{\varepsilon, A})\}\ \le\ 2n_{l}\ \le\ m_k\ \le\ m_k \| x\|. 
$$
Moreover, 
$$
D_k(\Ind_{\varepsilon, A})\ \le\ \frac{\sqrt{n_k}}{m_k}\ \le\  \frac{\sqrt{n_k}}{m_k}\| x\|.
$$
To find an upper bound for $Q_k(\Ind_{\varepsilon, A})$, let $A_1:=A\cap I_{n_k}$ and $A_2:=A\cap (n_k+I_{n_k})$. We have
$$
Q_k(\Ind_{\varepsilon, A})\ \le\ Q_k(\Ind_{\varepsilon, A_1})+Q_k(\Ind_{\varepsilon, A_2})\ \le\ \frac{\sqrt{n_k}}{m_k}\|x\|+Q_k(\Ind_{\varepsilon, A_2}).
$$
If $|A_2|>0$, we proceed by case analysis.

Case 1: If $| B\cap I_{2n_k}^c| \ge |A_2|/4$, then 
$$
\| x\|\ \ge\ D_k(x)\ \ge\ D_k(P_B(x))\ \ge\ \frac{\sqrt{|A_2|}}{2m_k}\ \ge\ \frac{Q_k(\Ind_{\varepsilon, A_2})}{2\sqrt{|A_2|}}\ \ge\ \frac{Q_k(\Ind_{\varepsilon, A_2})}{2\sqrt{n_k}}. 
$$

Case 2: If $| B_1:=B\cap I_{n_k}| \ge |A_2|/4$, choose $A_3\subset A_2$ and $B_3\subset B_1$ so that $|A_3|= |B_3|\ge |A_2|/4$ and define 
$$
\delta_n\ =\ 
\begin{cases}
\frac{| \xx_n^*(x)|}{\sqrt{n_k}\xx_n^*(x)}, & \mbox{ if } n\in B_3;\\
0, & \mbox{ if } n\in I_{n_k}\backslash B_3;\\
0, & \mbox{ if } n\in (n_k+I_{n_k})\backslash A_3. 
\end{cases}
$$
Since $|A_3| = |B_3|$, one can choose $(\delta_n)_{n\in A_3}$ so that $\delta=(\delta_n)_{n=1}^{2n_k}\in \Delta_k$. It follows from $A\cap \supp(x) = \emptyset$ that 
$$
\|x\|\ \ge\ Q_k(x)\ \ge\ \frac{1}{m_k} \sum_{n=1}^{2n_k}\delta_n \xx_n^*(x)\ \ge \ \frac{|A_2|}{4m_k\sqrt{n_k}}\ \ge\  \frac{Q_k(\Ind_{\varepsilon, A_2})}{4\sqrt{n_k}}.
$$

Case 3: If $| B_2:=B\cap (n_k+I_{n_k})| \ge |A_2|/2$, choose $A_4\subset A_2$ and $B_4\subset B_2$ with $|A_4|=|B_4|\ge |A_2|/2$. Let $\sigma: A_4\rightarrow B_4$ be a bijection and define for $1\le n\le 2n_k$,
$$
\delta_n\ =\ 
\begin{cases}
\frac{| \xx_n^*(x)|}{\xx_n^*(x)}, &\mbox{ if } n\in B_4;\\
-\frac{| \xx_{\sigma(n)}^*(x)|}{\xx_{\sigma(n)}^*(x)}, &\mbox{ if } n\in A_4;\\
0, &\mbox{ if } n\in (I_{2n_k})\backslash (B_4\cup A_4). 
\end{cases}
$$
Then $\delta=(\delta_n)_{n\in I_{2n_k}}\in \Delta_k$, so
$$
\|x\|\ \ge\ Q_k(x)\ \ge\ \frac{1}{m_k} \sum_{n=1}^{2n_k}\delta_n \xx_n^*(x)\ \ge\ \frac{|A_2|}{2m_k}\ \ge\ \frac{Q_k(\Ind_{\varepsilon, A_2})}{2}.
$$

We have considered all cases because $|B| \ge |A_2|$, and the combination of the above estimates yield 
\begin{equation}
\| \Ind_{\varepsilon, A}\|\ \le\ 5\sqrt{n_k}\|x\|,\label{lambdad}
\end{equation}
so 
$$\la_{n_k}^{d}\ \le\ 5\sqrt{n_k}.$$
For $\varepsilon>0$, it follows from \eqref{mu1} and \eqref{lambdad} that 
$$
\frac{\boldsymbol\mu_{2n_k}}{ (\la_{2n_k}^d)^{2-\varepsilon}}\ \ge\ \frac{n_k^{\varepsilon/2}}{100m_k^2} \ \xrightarrow[k\to \infty]{}\ \infty.  
$$
This completes the proof of \eqref{toprovemumlabdaalt} and of the proposition. 
\end{proof}
It follows at once from definitions that 
$$
\boldsymbol{\mu}_m\ \le\ \tilde{\boldsymbol{\mu}}_m\ \le\  \la_m \quad \mbox{ and }\ \boldsymbol{\mu}_m\ \le\ \boldsymbol{\nu}_m.
$$
Additionally, by \cite[Lemma 4.1, Inequalities (8.6), and Question 8.4]{AAB2}, 
$$
\max\{ \la_m, \boldsymbol{\nu}_m\} \ \lesssim\ ( \la_m^d)^2.
$$
Combining these inequalities with Proposition \ref{propositionmulambda}, we obtain the following.

\begin{corollary}\label{corollaryRMetcmulambda} We have
\begin{itemize}
    \item $\mathfrak{R}_B(\la_{m}, \la^{d}_m) = \mathfrak{R}_M(\la_{m}, \la^{d}_m) = 2$;
    \item $\mathfrak{R}_B(\boldsymbol\mu_{m}, \la^{d}_m) = \mathfrak{R}_M(\boldsymbol\mu_{m}, \la^{d}_m) = 2$;
    \item $\mathfrak{R}_B(\tilde{\boldsymbol\mu}_{m}, \la^{d}_m) = \mathfrak{R}_M(\tilde{\boldsymbol\mu}_{m}, \la^{d}_m) = 2$;
    \item $\mathfrak{R}_B(\boldsymbol{\nu}_{m},\la_m^d) = \mathfrak{R}_M(\boldsymbol{\nu}_{m},\la_m^d) = 2$.
\end{itemize}
\end{corollary}

\begin{remark}\label{remarkQ4}\normalfont Corollary \ref{corollaryRMetcmulambda} shows that none of the estimates in \cite[Question 8.4]{AAB2} can be improved. 
\end{remark}

\section{\texorpdfstring{The ratios with \(\la_m\), \(\mathbf L_m^a\), and \(\bm{\nu}_m\)}{The ratio with la_m, L_m^a, and nu_m}}\label{S6}
Recall from \cite[Page 23]{AAB2} that
\begin{equation}\label{e1}
\la_m\ \lesssim\  
\begin{cases}
(\mathbf L_m^a)^2, & \mbox{ if } p=1;\\
(\mathbf L_m^a)^{\left(2+\frac{1}{p}\right)}, &\mbox{ if } p<1.
\end{cases}
\end{equation}

The authors asked whether these bounds can be improved and, in particular, whether the inequality $\la_m \lesssim \mathbf L_m^a$ holds. They conjectured that the answer was negative (see \cite[Question 8.3]{AAB2} and its discussion). 

In this section, we first prove that for $p\in (0,1]$, $\la_m\lesssim (\mathbf L_m^a)^2$, which improves \eqref{e1} in the case $p < 1$. We then show that while $\mathfrak R_S(\la_m, \mathbf L_m^a) = 1$, it holds that
$\mathfrak R_B(\la_m, \mathbf L_m^a) = \mathfrak R_M(\la_m, \mathbf L_m^a) = 2$.

\begin{lemma}\label{lemmanotp}Let $\XB$ be a fundamental minimal system for a $p$-Banach space $X$. Then there is $C = C(p)\ge 1$ such that 
$$\la_m\ \le\ \tilde{\boldsymbol\mu}_m^d \la_m^d\ \le\  C(p) (\mathbf L_m^a)^2.$$ 
\end{lemma}
\begin{proof}
Given $x\in \XX$ with finite support, $A, B\subset \NN$, $|A|=|B|=k\le m$, $B\in G(x,k)$, and $\varepsilon\in \EE_A$, choose $D>A\cup \supp(x)$ with $|D|=k$. We have  
$$
\min_{n\in B}| \xx_n^*(x)|\| \Ind_{\varepsilon, A}\|\ \le\ \tilde{\boldsymbol\mu}_m^d \min_{n\in B}| \xx_n^*(x)|\| \Ind_{D}\|\ \le\  \tilde{\boldsymbol\mu}_m^d\la_m^d \| x\|. 
$$
Hence, $\la_m\le \tilde{\boldsymbol\mu}_m^d \la_m^d$. 

It is clear from definitions that $\tilde{\boldsymbol\mu}_m^d\le \mathbf L_m^a$, and it follows from \cite[Theorem 4.2]{AAB2} that $\la_m^d\le C(p) \mathbf L_m^a$, so the proof is complete for finitely supported elements. An application of \cite[Lemma 3.4]{BBG2025} allows us to extend the result to arbitrary $x\in \XX$. 
\end{proof}

To compute the ratios in the section title and answer \cite[Question 8.3]{AAB2}, we construct a Markushevich basis for which 
\begin{equation}
\limsup_{m\to \infty}\frac{\la_m}{(\mathbf L_m^a)^{2-\varepsilon}}\ =\ \infty,\label{toprovelmalabda}
\end{equation}
which proves the optimality of $\la_m\lesssim (\mathbf L_m^a)^2$.

For the proof, we  use \eqref{ch2} and modify the construction in the proof of Proposition \ref{propositionmulambda}. The proof is slightly more involved as we need to give upper bounds for the quasi-greedy parameters. 

\begin{proposition}\label{propositionLmalambda}There is a Banach space $\XX$ with a Markushevich basis $\XB$ for which 
\begin{equation*}
\limsup_{m\to \infty}\frac{\la_m}{(\mathbf L_m^a)^{2-\varepsilon}}\ =\ \infty.
\end{equation*}
\end{proposition}
\begin{proof}
Choose two sequences $(n_k)_{k = 1}^\infty\subset \NN$ and $(m_k)_{k = 1}^\infty\subset \NN_{>5}$ such that, for each $k\in \NN$,
$$
m_k^{4k}\ <\ n_k\ <\ \frac{m_{k+1}}{2}.
$$
Set $n_0=m_0 =0$ and define on $\mathtt{c}_{00}$
$$
D_k((a_n)_{n\in\NN})\ :=\ \frac{1}{m_k}\sup_{\substack{A\subset \NN\\ A>2n_k\\ 1\le |A|\le n_k}}\frac{1}{\sqrt{|A|} }\sum_{n\in A}|a_n|, \quad k\in\mathbb{N}. 
$$

 Let $\Delta_k$ be the set of finite sequences $\delta=(\delta_n)_{1\le n\le 2n_k}$ with the following properties: 
\begin{enumerate}[\rm 1.]
\item $\left|\sum_{n=1}^{n_k}\left(\delta_n+ \frac{1}{\sqrt{n_k}}\delta_{n+n_k}\right) \right|\le 1$. 
\item For each $1\le n\le 2n_k$, $ | \delta_n|\le 1$. 
\item For each $A\subset I_{n_k}$, $\sum_{n\in A}| \delta_n| \le \sqrt{|A|}$. 
\end{enumerate}
Define on $\mathtt{c}_{00}$  
$$
Q_k((a_n)_{n\in\NN})\ :=\ \frac{1}{m_k}\left| \sup_{\delta \in \Delta_k}\sum_{n=1}^{2n_k}\delta_n  a_n \right|. 
$$

Let $\XX$ be the completion of $\mathtt{c}_{00}$ with the norm
$$
\| (a_n)_{n\in\NN}\|\ :=\ \sup_{k\in \NN}\{| a_k|, D_k((a_n)_{n\in\NN}), Q_k((a_n)_{n\in\NN})\}
$$
and let $\XB$ be the canonical unit vector system of $\XX$.  Note that $\XB$ and $\XB^*$ are both normalized and that $\XB$ is a Markushevich basis by Lemma \ref{lemmaMarkushevich}. 

To see that \eqref{toprovelmalabda} holds, by \cite[Theorem 4.2]{AAB2}, it suffices to show that for each $\varepsilon>0$, 
\begin{align}
\lim_{k\to \infty}\frac{\la_{n_k}}{\max\{ \la_{n_k}^d, \mathbf g_{n_k}\}^{2-\varepsilon}}\ =\ \infty.\label{toprovelmalabdaalt}
\end{align}
To this end, first we claim that
\begin{equation}
Q_k(x)\ =\ \frac{1}{m_k}\sum_{n\in \NN}| \xx_n^*(x)|, \quad x\in \spn\{ \xx_n: n\in \{n_k + 1, \ldots, 2n_k\}\}\label{claim1}
\end{equation}
and 
\begin{equation}
Q_k(x)\ \ge\ \frac{1}{m_k\sqrt{| n_k|}}\sum_{n\in \NN}| \xx_n^*(x)|, \quad x\in \spn\{ \xx_n: n\in I_{n_k}\}.\label{claim2}
\end{equation}
To prove \eqref{claim1}, first note that the inequality
$$
Q_k(x)\ \le \ \frac{1}{m_k}\sum_{n\in \NN}| \xx_n^*(x)|
$$
is immediate. Now for $1\le n\le 2n_k$, define
$$
\delta_{n}\ :=\
\begin{cases}
\frac{|\xx_n^*(x)|}{\xx_n^*(x)}, & \mbox{ if } n\in \supp(x); \\
-\frac{|\xx_{n+n_k}^*(x)|}{\sqrt{n_k}\xx_{n+n_k}^*(x)}, &\mbox{ if } n+n_k\in \supp(x);\\
0, & \mbox{ otherwise}. 
\end{cases}
$$
It is routine to check that $\delta=(\delta_j)_{1\le j\le 2n_k}\in \Delta_k$ and that
$$
\sum_{n=1}^{2n_k}\delta_n \xx_n^*(x)\ =\ \sum_{n\in \NN}| \xx_n^*(x)|, 
$$
which completes the proof of \eqref{claim1}. 

The proof of \eqref{claim2} is similar: for $1\le n\le 2n_k$, we define
$$
\delta_{n}\ :=\
\begin{cases}
\frac{| \xx_n^*(x)|}{\sqrt{n_k}\xx_n^*(x)},& \mbox{ if } n\in \supp(x); \\
-\frac{| \xx_{n-n_k}^*(x)|}{\xx_{n-n_k}^*(x)}, &\mbox{ if } n-n_k\in \supp(x);\\
0, & \mbox{ otherwise},
\end{cases}
$$
and \eqref{claim2} follows. 

For $k>1$, we find a lower bound for $\la_{n_k}$. Let 
$$x\ :=\ \Ind_{I_{n_k}}+\frac{1}{\sqrt{n_k}}\Ind_{\{n_k+1, \ldots, 2n_k\}}.$$ 
Then 
$$
Q_k(x)\ \le\ \frac{1}{m_k}\quad \mbox{ and }\quad D_k(x)\ =\ 0. 
$$
For $1\le l<k$, 
$$
\max\{D_l(x), Q_l(x)\}\ \le\ 2n_{k-1}\ \le\ m_k, 
$$
whereas for $l>k$, 
$$
Q_l(x)\ \le\ \frac{2n_k}{m_{l}}\ \le\ 1 \quad \mbox{ and }\quad D_l(x)\ =\ 0. 
$$
On the other hand, by \eqref{claim1}, we have
$$
\|\Ind_{\{n_k+1, \ldots, 2n_k\}}\|\ \ge\ Q_k(\Ind_{\{n_k+1, \ldots, 2n_k\}})\ =\ \frac{n_k}{m_k}. 
$$
Since $I_{n_k}\in G(x,n_k)$, it follows from  the above estimates  that
$$
\la_{n_k}\ \ge\ \frac{n_k}{m_k^2}. 
$$

Next, we find an upper bound for $\la_{n_k}^d$. To this end, fix $x\in \XX$, $1\le m\le n_k$, $B\in G(x,m)$, $A\in \NN^{(m)}$ with $\supp(x)\cap A=\emptyset$, and $\varepsilon\in \EE_A$. We may assume that $\min_{n\in B}|\xx_n^*(x)|=1$. For each $l>k$, we have
$$
\max\{ Q_l(\Ind_{\varepsilon, A}), D_l(\Ind_{\varepsilon, A})\}\ \le\ \frac{n_k}{m_l}\ \le\ 1\ \le\ \|x\|. 
$$
On the other hand, for $1\le l<k$, 
$$
\max\{Q_l(\Ind_{\varepsilon, A}), D_l(\Ind_{\varepsilon, A})\}\ \le\ 2n_{l}\ \le\ m_k\ \le\ m_k \|x\|\ \le\ \sqrt{n_k} \|x\|. 
$$
Also, 
$$
D_k(\Ind_{\varepsilon, A})\ \le\ \frac{\sqrt{n_k}}{m_k}\ \le\  \frac{\sqrt{n_k}}{m_k}\|x\|.
$$
To find an upper estimate for $Q_k(\Ind_{\varepsilon, A})$, let $A_1:=A\cap I_{n_k}$ and $A_2:=A\cap \{n_k + 1, \ldots, 2n_k\}$. We have
$$
Q_k(\Ind_{\varepsilon, A})\ \le\ Q_k(\Ind_{\varepsilon, A_1})+Q_k(\Ind_{\varepsilon, A_2})\ \le\ \frac{\sqrt{n_k}}{m_k}\|x\|+Q_k(\Ind_{\varepsilon, A_2}).
$$
If $|A_2| >0$, we proceed by case analysis.

Case 1: If $| B\cap I_{2n_k}^c| \ge |A_2|/4$, then by \eqref{claim1},
$$
\|x\|\ \ge\ D_k(x)\ \ge\ D_k(P_B(x))\ \ge\ \frac{\sqrt{|A_2|}}{2m_k} \ =\ \frac{Q_k(\Ind_{\varepsilon, A_2})}{2\sqrt{|A_2|}}\ \ge\ \frac{Q_k(\Ind_{\varepsilon, A_2})}{2\sqrt{n_k}}. 
$$

Case 2: If $| B_1:=B\cap I_{n_k}| \ge |A_2|/4$, choose $A_3\subset A_2$ and $B_3\subset B_1$ so that $|A_3|=| B_3|\ge |A_2|/4$ and define 
$$
\delta_n\ :=\
\begin{cases}
\frac{| \xx_n^*(x)|}{\sqrt{n_k}\xx_n^*(x)}, & \mbox{ if } n\in B_3;\\
0, & \mbox{ if } n\in I_{n_k}\backslash B_3;\\
0, & \mbox{ if } n\in \{n_k+1, \ldots, 2n_k\}\backslash A_3. 
\end{cases}
$$
Given that $|A_3| = |B_3|$, one can define $(\delta_n)_{n\in A_3}$ so that $\delta=(\delta_n)_{n=1}^{2n_k}\in \Delta_k$. Then by \eqref{claim1},
$$
\|x\|\ \ge\ Q_k(x)\ \ge\ \frac{1}{m_k} \sum_{n=1}^{2n_k}\delta_n \xx_n^*(x)\ \ge\ \frac{|A_2|}{4m_k\sqrt{n_k}}
 \ =\ \frac{Q_k(\Ind_{\varepsilon, A_2})}{4\sqrt{n_k}}.
$$

Case 3: If $|B_2:=B\cap \{n_k + 1, \ldots, 2n_k\}| \ge |A_2|/2$, choose $A_4\subset A_2$ and $B_4\subset B_2$ with $|A_4|=|B_4|\ge |A_2|/2$ and let $\sigma: A_4\rightarrow B_4$ be a bijection. Define for $1\le n\le 2n_k$, 
$$
\delta_n\ :=\
\begin{cases}
\frac{| \xx_n^*(x)|}{\xx_n^*(x)}, &\mbox{ if } n\in B_4;\\
-\frac{| \xx_{\sigma(n)}^*(x)|}{\xx_{\sigma(n)}^*(x)}, &\mbox{ if } n\in A_4;\\
0, &\mbox{ if } n\in I_{2n_k}\backslash (A_4\cup B_4). 
\end{cases}
$$
Then $\delta=(\delta_n)_{n=1}^{2n_k}\in \Delta_k$, so
$$
\|x\|\ \ge\ Q_k(x)\ \ge\ \frac{1}{m_k} \sum_{n=1}^{2n_k}\delta_n \xx_n^*(x)\ \ge \ \frac{|A_2|}{2m_k}\ =\ \frac{Q_k(\Ind_{\varepsilon, A_2})}{2}.
$$

We have considered all cases, and combining the above estimates yields
$$
\| \Ind_{\varepsilon, A}\|\ \le\ 5\sqrt{n_k}\|x\|,
$$
so 
\begin{equation}\label{e3}\la_{n_k}^{d}\ \le\ 5\sqrt{n_k}.\end{equation}

It remains to establish an upper bound for $\mathbf g_{n_k}$. Fix $x\in \XX$, $1\le m\le n_k$, and $B\in G(x,m)$. We may assume that $\|x\|=1$, so $\max_n|\xx^*_n(x)|\le 1$. For each $l\in \NN$, 
$$
 D_l(P_B(x))\ \le\ D_l(x)\ \le\ \|x\|.  
$$
Now, suppose there is an $s\in \NN$ such that 
$$
\| P_B(x)\|\ =\ Q_s(P_B(x))
$$
If $s > k$, then 
$$\|P_B(x)\|\ =\ Q_s(P_B(x))\ \le\ \frac{n_k}{m_s}\ \le\ \frac{1}{2}\ =\ \frac{1}{2}\|x\|.$$
If $s< k$, then
$$\|P_B(x)\|\ =\ Q_s(P_B(x))\ \le\ 2n_s\ < \ m_k \ <\ \frac{\sqrt{n_k}}{m_k}\|x\|.$$
Suppose that $s = k$. Let $B_1:=B\cap I_{n_k}$ and $B_2:=B\cap \{n_k + 1, \ldots, 2n_k\}$. Then 
$$
\| P_B(x)\|\ =\ Q_k(P_B(x))\ \le\ Q_k(P_{B_1}(x))+Q_k(P_{B_2}(x)), 
$$
and since $\max_{n}|\xx_n^*(x)|\le 1$, 
\begin{align}
Q_k(P_{B_1}(x))&\ =\ \frac{1}{m_k}\left|\sup_{\delta \in \Delta_k}\sum_{j=1}^{2n_k}\delta_j  \xx_j^*(P_{B_1}(x))\right|\nonumber\\
&\ \le\ \frac{1}{m_k}\sum_{j = 1}^{n_k}| \delta_j \xx_j^*(x)|\ \le\ \frac{\sqrt{n_k}}{m_k}\ =\ \frac{\sqrt{n_k}}{m_k}\|x\|. \label{thefirstone}
\end{align}
If $Q_k(P_{B_2}(x))\le \frac{\sqrt{n_k}}{m_k}$, then we use the same upper bound in this case. Otherwise, \eqref{claim1} and the fact that $\max_{n}|\xx_n^*(x)|\le 1$ imply that $| B_2|>\sqrt{n_k}$. Hence, using again  \eqref{claim1}, $\max_{n}|\xx_n^*(x)|\le 1$, and  
$$
\frac{25}{m_k}\ <\ \frac{\sqrt{n_k}}{m_k}\ <\ Q_k(P_{B_2}(x)), 
$$
we can write $B_2=B_{2,1}\sqcup B_{2,2}$ with $B_{2,1}\in G(P_{B_2}(x))$ and 
\begin{equation}
Q_k(P_{B_{2,1}}(x))\ \ge\ Q_k(P_{B_{2,2}}(x))\ \ge\  Q_k(P_{B_{2,1}}(x))-\frac{2}{m_k}\ \ge\ \frac{1}{2}Q_k(P_{B_{2,1}}(x)).  \label{equalsplit}
\end{equation}
Let 
$$
\theta_1\ :=\ \min_{n\in B_{2,1}}| \xx_n^*(x)|, \quad \theta_2\ :=\ \max_{n\in B_{2,2}}| \xx_n^*(x)|, \quad
D\ :=\ \left\{ n\in I_{n_k}: | \xx_n^*(x)|\ \le\ \frac{\sqrt{n_k}}{2}\theta_1\right\}.  
$$
We proceed by case analysis.

Case 1: If $|D|\ge |B_{2,1}|$, choose $D_1\subset D$ so that $| D_1|=|B_{2,1}|$ and let $\sigma: D_1\rightarrow B_{2,1}$ be a bijection. For $1\le n\le 2n_k$, define 
$$
\delta_n\ :=\
\begin{cases}
\frac{| \xx_n^*(x)|}{\xx_n^*(x)}, & \mbox{ if }n\in B_{2,1};\\
-\frac{1}{\sqrt{n_k}}\frac{|\xx_{\sigma(n)}^*(x)|}{\xx_{\sigma(n)}^*(x)}, & \mbox{ if }n\in D_1;\\
0, & \mbox{ otherwise. } 
\end{cases}
$$
It is routine to check that $\delta = (\delta_j)_{1\le j\le 2n_k}\in \Delta_k$. Since 
$$
Q_k(P_{B_2}(x))\ \le\ 2 Q_k(P_{B_{2,1}}(x)), 
$$
we have 
\begin{align*}
m_k Q_k(x)&\ \ge\ \sum_{n=1}^{2n_k}\delta_n \xx_n^*(x)\ =\ \sum_{n\in B_{2,1}}| \xx_n^*(x)| -\frac{1}{\sqrt{n_k}}\sum_{n\in D_1}| \xx_{n}^*(x)|\\
&\ \ge\ \sum_{n\in B_{2,1}}| \xx_n^*(x)|-|D_1|\frac{\theta_1}{2}\\
&\ \ge\  \frac{1}{2}\sum_{n\in B_{2,1}}| \xx_n^*(x)|\ =\ \frac{m_k}{2}Q_k(P_{B_{2,1}}(x))\ \ge\ \frac{m_k}{4}Q_k(P_{B_{2}}(x)). 
\end{align*} 
Hence, 
$$
Q_k(P_{B_{2}}(x))\ \le\ 4Q_k(x)\ \le\ 4\|x\|.  
$$

Case 2: If $|D|<|B_{2,1}|$, let $D':=I_{n_k}\backslash D$. It follows from \eqref{claim1} and \eqref{equalsplit} that 
$$\theta_1|B_{2,1}|\ \le\ \sum_{n\in B_{2,1}}|\xx^*_n(x)|\ \le\ 2\sum_{n\in B_{2,2}}|\xx_n^*(x)|\ \le\ 2|B_{2,2}|\theta_2\ \le\ 2\theta_1|B_{2,2}|,$$
so $|B_{2,1}|\le 2|B_{2,2}|$. As $|B_{2,1}|+|B_{2,2}|\le n_k$, this entails that $|B_{2,1}|\le 2n_k/3$. Hence, 
$$|D'|\ =\ n_k-|D|\ >\ n_k-|B_{2,1}|\ \ge\ \frac{n_k}{3}.$$ 
Moreover, for each $n\in D'$, 
\begin{equation}\label{e2}
|\xx_n^*(x)|\ >\ \frac{\sqrt{n_k}\theta_1}{2}\ >\ \theta_1. 
\end{equation}
Since $B\in G(x)$ and $\min B\le \theta_1$, \eqref{e2} implies that $D'\subset B_1$. Hence, \eqref{claim1}, \eqref{claim2}, and \eqref{equalsplit} yield
\begin{align*}
Q_k(P_{B_1}(x))&\ \ge\ \frac{1}{m_k\sqrt{n_k}}\sum_{n\in B_1}| \xx_n^*(x)|\\
&\  \ge\ \frac{1}{m_k\sqrt{n_k}}\sum_{n\in D'}| \xx_n^*(x)|\\
&\ \ge\ \frac{\theta_1}{2m_k}|D'|\ \ge\ \frac{\theta_1 n_k}{6m_k}\\
&\ \ge\ \frac{1}{6}Q_k(P_{B_{2,2}}(x))\ \ge\ \frac{1}{18}Q_k(P_{B_2}(x)).
\end{align*}
Therefore, \eqref{thefirstone} gives
$$
Q_k(P_{B_2}(x))\ \le\ \frac{18\sqrt{n_k}}{m_k}\|x\|. 
$$
Combining all of the above, we obtain
\begin{equation}\label{e4}
\| P_B(x)\|\ \le\  \frac{19\sqrt{n_k}}{m_k}\|x\|\ \Longrightarrow \  \mathbf g_{n_k}\ \le\ \frac{19\sqrt{n_k}}{m_k}.
\end{equation}
It follows from \eqref{e3} and \eqref{e4} that
$$
\max\{ \mathbf g_{n_k}, \la_{n_k}^d\} \ \le\  5\sqrt{n_k}. 
$$

For $\varepsilon>0$, we have 
$$
\frac{\la_{m_k}}{\max\{ \mathbf g_{n_k}, \la_{n_k}^d\}^{2-\varepsilon}}\ \ge\ \frac{n_k^{\varepsilon/2}}{25m_k^2} \ \xrightarrow[k\to \infty]{}\ \infty.  
$$
This completes our proof.
\end{proof}

\begin{corollary}\label{corollaryRMetc}We have $\mathfrak{R}_M(\la_{m}, \mathbf L_m^a)=\mathfrak{R}_B(\la_{m}, \mathbf L_m^a)=2$.
\end{corollary}
\begin{proof}
This follows from Lemma \ref{lemmanotp} and Proposition \ref{propositionLmalambda}. 
\end{proof}

We close this section by computing $\mathfrak{R}_S(\la_{m}, \mathbf L_m^a)$ and several ratios that involve $\bm{\nu}_m$.

\begin{proposition}
We have $\mathfrak{R}_S(\la_m, \mathbf L_m^a) = 1$.     
\end{proposition}

\begin{proof}
    By Theorem \ref{t3v3}, 
    $$\frac{\la_m}{\mathbf L_m^a}\ \lesssim\ \frac{\la_m^d}{\mathbf L_m^a}\ \lesssim\ 1,$$
    so $\mathfrak{R}_S(\la_m, \mathbf L_m^a) \le 1$. We show that $\mathfrak{R}_S(\la_m, \mathbf L_m^a)\ge 1$. Suppose, for a contradiction, that $\mathfrak{R}_S(\la_m, \mathbf L_m^a) < 1$. Then there is $\varepsilon\in (0,1)$ with 
    $$\frac{\la_m}{(\mathbf L_m^a)^\varepsilon}\ \lesssim\ 1.$$
    Hence, by Theorem \ref{t3v3} and \eqref{ch2}, \begin{equation}\label{e5}\la_m^d\ \lesssim\ (\mathbf L_m^a)^\varepsilon\ \lesssim\ \max\{(\la_m^d)^\varepsilon, (\mathbf g_m)^\varepsilon\}.\end{equation}
    Let $\XX$ be a space with a Schauder basis $\mathcal{X}$ that is quasi-greedy but not democratic. Then $$\lim_{m\rightarrow\infty}\la_m^d \ =\ \infty\quad \mbox{ and }\quad \mathbf g_m \ \lesssim\ 1.$$ Hence, 
    \begin{equation}\label{e6} \max\{(\la_m^d)^\varepsilon, (\mathbf g_m)^\varepsilon\}\ \approx\ (\la_m^d)^\varepsilon.\end{equation}
    By \eqref{e5} and \eqref{e6}, $\la_m^d \lesssim (\la_m^d)^{\varepsilon}$, which contradicts that 
    $\la_m^d\rightarrow\infty$. 
\end{proof}

One can find in the literature Schauder bases for which $\tilde{\boldsymbol\mu}_m$ is bounded but $\boldsymbol{\nu}_m$ is not (see for instance, \cite[Example 3.15]{BDKOW} or \cite[Theorem 4.4]{AAB2024}).  Finally, there are Schauder bases for which $\boldsymbol{\nu}_m$ is bounded but $\la_m^d$ is not (see \cite[Theorem 3.1]{AAB2024} and \cite[Proposition 5.3]{AABW2021}). It follows that
\begin{align*}
\mathfrak{R}_B(\la^d_{m}, \boldsymbol{\nu}_m)&\ =\ \mathfrak{R}_M(\la^d_{m}, \boldsymbol{\nu}_m)\ =\ \mathfrak{R}_S(\la^d_{m}, \boldsymbol{\nu}_m)\ =\ \infty;\\
\mathfrak{R}_B(\boldsymbol{\nu}_m, \boldsymbol\mu_m)&\ =\ \mathfrak{R}_M(\boldsymbol{\nu}_m, \boldsymbol\mu_m) \ =\ \mathfrak{R}_S(\boldsymbol{\nu}_m, \boldsymbol\mu_m)\ =\ \infty.
\end{align*}

\begin{proposition}\label{propnula}
    We have
    $$\mathfrak{R}_B(\boldsymbol{\nu}_m, \la_{m})\ =\ \mathfrak{R}_M(\boldsymbol{\nu}_m, \la_{m})\ =\ \mathfrak{R}_S(\boldsymbol{\nu}_m, \la_{m})=1.$$
\end{proposition}

\begin{proof}
    By \cite[Proposition 8.2]{AAB2}, we have $\mathfrak{R}_B(\boldsymbol{\nu}_m, \la_{m})\le 1$. Suppose, for a contradiction, that $\mathfrak{R}_S(\boldsymbol{\nu}_m, \la_{m}) < 1$, i.e., there is an $\varepsilon\in (0,1)$ such that for all Schauder bases, 
    $$\frac{\boldsymbol{\nu}_m}{(\la_{m})^\varepsilon}\ \lesssim \ 1.$$
    By \cite[(5.3)]{AAB2}, 
    \begin{equation}\label{e7}\la_m    \ \lesssim\ \ \mathbf r_m \bm{ \mu}_m \ \Longrightarrow\  \frac{\bm \nu_m}{ \mathbf r^{\varepsilon}_m \bm{ \mu}^\varepsilon_m} \lesssim\ 1\ \Longrightarrow\ \frac{\bm \nu_m^{1-\varepsilon}}{\mathbf r^{\varepsilon}_m}\ \lesssim\ 1.\end{equation}
    Let $\XX$ be a space with a quasi-greedy Schauder basis $\XB$ that is not democratic (see \cite[Proposition 6.10]{BDKOW}. Then $\sup_m \bm{\nu}_m = \infty$, while $\mathbf r_m \lesssim 1$. This contradicts \eqref{e7}.
    Therefore, 
    $$1\ \le\ \mathfrak{R}_S(\boldsymbol{\nu}_m, \la_{m}) \ \le\ \mathfrak{R}_M(\boldsymbol{\nu}_m, \la_{m})\ \le\ \mathfrak{R}_B(\boldsymbol{\nu}_m, \la_{m})\ \le\ 1,$$
    and we are done. 
\end{proof}

The next proposition follows immediately from Proposition \ref{propnula} and Theorem \ref{m4}. 

\begin{proposition}
    We have
$$\mathfrak{R}_B(\boldsymbol{\nu}_m, \la^d_{m})\ =\ \mathfrak{R}_M(\boldsymbol{\nu}_m, \la^d_{m})\ =\ 2 \mbox{ and }\mathfrak{R}_S(\boldsymbol{\nu}_m, \la^d_{m})\ =\ 1.$$
\end{proposition}

%%%%%%%%%%%%%%%%%%%%%%%%%%%%%%%%%%%%%%%%%%%%%%%%%%%%%%%%%%%%%%%%%%%%%%%%%%%%%%%%%%%%%%%%%%%%%%%%%%%%%%%%%%%%%%%%%%%%%%%%%%%%%%%%%%%%%%%%%%%%%%%%%%%%%%%%%%%%%%%%%%%%%%%%%%%%%%%%%%%%%%%%%%%%%%%%%%%%%%%%%%%%%%%%%%%%%%%%%%%%%%%%%%%%%%%%%%%%%%%%%%%%%%%%%%%%%%%%%%%%%%%%%%%%%%%%%%

\end{document}